\documentclass[final,notitlepage,12pt,reqno]{amsart}

\UseRawInputEncoding
\usepackage{graphicx,extarrows}
\usepackage{calc}
\usepackage{url}
\newlength{\depthofsumsign}
\newcommand{\nsum}[1][1.0]{
   \mathop{%
        \raisebox
            {-#1\depthofsumsign+1\depthofsumsign}
            {\scalebox
                {#1}
                {$\displaystyle\sum$}%
            }
    }
}

\makeatletter
\let\I\@undefined
\makeatother

\usepackage{geometry}
\usepackage{indentfirst}
\usepackage[normalem]{ulem}
\usepackage{float}
\usepackage{amsthm}

\usepackage{enumitem}[2011/09/28]
\setenumerate{align=left, leftmargin=0pt,labelsep=.5em, labelindent=0\parindent,listparindent=\parindent,itemindent=*}
\usepackage{array}
\usepackage[all]{xy}

\usepackage{exscale,relsize}
\usepackage{multirow}

\usepackage{caption}[2012/02/19]

\usepackage{longtable,lscape}

\usepackage{colonequals}

\usepackage{amssymb,bm,amsmath}
\usepackage{amsfonts}

\usepackage[OT2,T1,T2A]{fontenc}

\usepackage[french,german,romanian,russian,english]{babel}
\usepackage{appendix}

\newcommand{\qoppa}{\kern-.1em\rotatebox{180}{\raisebox{-0.42em}{$'$}}\kern-.2em{o}}

\DeclareMathOperator{\ELi}{ELi}

\DeclareMathOperator{\Li}{Li}

\DeclareMathOperator{\am}{am}

\DeclareMathOperator{\D}{d}
\DeclareMathOperator{\I}{Im}
\DeclareMathOperator{\RE}{Re}

\def\eor{\hfill$ \square$}

\newcolumntype{L}{>{$}l<{$}}
\newcolumntype{C}{>{$}c<{$}}
\newcolumntype{R}{>{$}r<{$}}

\theoremstyle{plain}
\newtheorem{theorem}{Theorem}[section]
\newtheorem{proposition}[theorem]{Proposition}
\newtheorem{lemma}[theorem]{Lemma}

\newenvironment{remark}[1][Remark]{\begin{trivlist}
\item[\hskip \labelsep {\bfseries #1}]}{\end{trivlist}}

\theoremstyle{definition}

\numberwithin{equation}{section}

\usepackage{Baskervaldx}
\usepackage[baskervaldx]{newtxmath}
\usepackage[cal=cm,scr=rsfs,frak=euler]{mathalfa}

\usepackage[counter,user]{zref}
\usepackage{refcount}
\makeatletter

\zref@newprop{AAAX}{\the\value{\zref@getcurrent{counter}}}%

\newcommand\AAAXlabel[1]{%
  \zref@labelbyprops{#1}{AAAX}%
  \label{#1}%
}

\zref@newprop{BBBX}{\the\value{\zref@getcurrent{counter}}}%

\newcommand\BBBXlabel[1]{%
  \zref@labelbyprops{#1}{BBBX}%
  \label{#1}%
}

\zref@newprop{AAAY}{\the\value{\zref@getcurrent{counter}}}%

\newcommand\AAAYlabel[1]{%
  \zref@labelbyprops{#1}{AAAY}%
  \label{#1}%
}

\zref@newprop{BBBY}{\the\value{\zref@getcurrent{counter}}}%

\newcommand\BBBYlabel[1]{%
  \zref@labelbyprops{#1}{BBBY}%
  \label{#1}%
}

\zref@newprop{AAAA}{\the\value{\zref@getcurrent{counter}}}%

\newcommand\AAAAlabel[1]{%
  \zref@labelbyprops{#1}{AAAA}%
  \label{#1}%
}

\zref@newprop{BBBB}{\the\value{\zref@getcurrent{counter}}}%

\newcommand\BBBBlabel[1]{%
  \zref@labelbyprops{#1}{BBBB}%
  \label{#1}%
}

\zref@newprop{AAA}{\the\value{\zref@getcurrent{counter}}}%

\newcommand\AAAlabel[1]{%
  \zref@labelbyprops{#1}{AAA}%
  \label{#1}%
}

\zref@newprop{BBB}{\the\value{\zref@getcurrent{counter}}}%

\newcommand\BBBlabel[1]{%
  \zref@labelbyprops{#1}{BBB}%
  \label{#1}%
}

\zref@newprop{AA}{\the\value{\zref@getcurrent{counter}}}%

\newcommand\AAlabel[1]{%
  \zref@labelbyprops{#1}{AA}%
  \label{#1}%
}

\zref@newprop{BB}{\the\value{\zref@getcurrent{counter}}}%

\newcommand\BBlabel[1]{%
  \zref@labelbyprops{#1}{BB}%
  \label{#1}%
}

\zref@newprop{CC}{\the\value{\zref@getcurrent{counter}}}%

\newcommand\CClabel[1]{%
  \zref@labelbyprops{#1}{CC}%
  \label{#1}%
}

\zref@newprop{DD}{\the\value{\zref@getcurrent{counter}}}%

\newcommand\DDlabel[1]{%
  \zref@labelbyprops{#1}{DD}%
  \label{#1}%
}

\zref@newprop{EE}{\the\value{\zref@getcurrent{counter}}}%

\newcommand\EElabel[1]{%
  \zref@labelbyprops{#1}{EE}%
  \label{#1}%
}

\zref@newprop{SS}{\the\value{\zref@getcurrent{counter}}}%

\newcommand\SSlabel[1]{%
  \zref@labelbyprops{#1}{SS}%
  \label{#1}%
}

\zref@newprop{SSO}{\the\value{\zref@getcurrent{counter}}}%

\newcommand\SSOlabel[1]{%
  \zref@labelbyprops{#1}{SSO}%
  \label{#1}%
}

\zref@newprop{SSSO}{\the\value{\zref@getcurrent{counter}}}%

\newcommand\SSSOlabel[1]{%
  \zref@labelbyprops{#1}{SSSO}%
  \label{#1}%
}

\zref@newprop{HHHA}{\the\value{\zref@getcurrent{counter}}}%

\newcommand\HHHAlabel[1]{%
  \zref@labelbyprops{#1}{HHHA}%
  \label{#1}%
}

\zref@newprop{HHHB}{\the\value{\zref@getcurrent{counter}}}%

\newcommand\HHHBlabel[1]{%
  \zref@labelbyprops{#1}{HHHB}%
  \label{#1}%
}

\makeatother
\usepackage[scr=rsfs]{mathalfa}
\DeclareMathAlphabet{\mathsf}{OT1}{\sfdefault}{m}{n}
\SetMathAlphabet{\mathsf}{bold}{OT1}{\sfdefault}{m}{n}
\DeclareSymbolFontAlphabet{\mathbb}{AMSb}

\usepackage{color}

\def\bg{\bigg}
\def\({\bg(}
\def\){\bg)}

\begin{document}

\pagenumbering{roman}
\selectlanguage{english}
\title[Series involving  binomial coefficients  and harmonic numbers]{Series involving central binomial coefficients \\ and higher-order harmonic numbers}
\author{Zhi-Wei Sun}\address[Z.-W.\ Sun]{School of Mathematics, Nanjing
University, Nanjing 210093, People's Republic of China}
\email{{\tt zwsun@nju.edu.cn}
\newline\indent
{\it Homepage}: {\tt http://maths.nju.edu.cn/\lower0.5ex\hbox{\~{}}zwsun}}
 \author{Yajun Zhou}
\address[Y. Zhou]{Program in Applied and Computational Mathematics (PACM), Princeton University, Princeton, NJ 08544} \email{yajunz@math.princeton.edu}\curraddr{\textrm{} \textsc{Academy of Advanced Interdisciplinary Studies (AAIS), Peking University, Beijing 100871, P. R. China}}\email{yajun.zhou.1982@pku.edu.cn}
\date{\today}\thanks{\textit{Keywords}:  Binomial coefficients, harmonic numbers, modular forms, Epstein zeta functions, special $L$-values\\\indent\textit{MSC 2020}: 11B65, 11F67, 11M06; 05A19\\\indent * Z.-W. Sun was supported by the Natural Science Foundation of China (grant no.\ 12371004). Y. Zhou was supported in part  by the Applied Mathematics Program within the Department of Energy
(DOE) Office of Advanced Scientific Computing Research (ASCR) as part of the Collaboratory on
Mathematics for Mesoscopic Modeling of Materials (CM4)}

\begin{abstract}We derive  modular  parametrizations for certain infinite series whose summands involve central binomial coefficients and higher-order harmonic numbers. When the  rates of  convergence  are certain rational numbers, modularity allows us to reduce the corresponding series to special values of the Dirichlet  $L$-functions.
For example, we establish the following identities  conjectured by Sun:\[\sum_{k=0}^\infty\binom{2k}{k}^3\left[ \mathsf H_{2k}^{(2)}-\frac{25}{92}\mathsf H_{ k}^{(2)} +\frac{735L_{-7}(2)-86\pi^{2}}{1104}\right]\frac{1}{4096^{k}}=0,\]\[\sum_{k=0}^\infty\binom{2k}k^3\left[\mathsf H_{2k}^{(3)}-\frac{43}{352}\mathsf H_k^{(3)}\right]\frac{42k+5}{4096^k}=\frac{555\zeta(3)}{77\pi}-\frac{32G}{11},\]    where $ \mathsf H^{(r)}_k\colonequals \sum_{0<n\leq k}\frac{1}{n^r}$, $ L_{-7}(2)\colonequals \sum_{n=1}^\infty\left(\frac{-7}{n}\right)\frac{1}{n^2}=\frac{1}{1^2}+\frac{1}{2^2}-\frac{1}{3^2}+\frac{1}{4^{2}}-\frac{1}{5^{2}}-\frac{1}{6^{2}}+\frac{1}{8^{2}}+\cdots $, $ G\colonequals \sum_{n=0}^\infty\frac{(-1)^n}{(2n+1)^2}$, and $ \zeta(3)\colonequals \sum_{n=1}^\infty\frac1{n^3}$.  \end{abstract}

\maketitle
\pagenumbering{arabic}

\section{Introduction}The four known classical rational Ramanujan-type series (see \cite[Chapter 14]{Cooper2017Theta}) involving cubes of central binomial coefficients $ \binom{2k}k\colonequals \frac{(2k)!}{(k!)^2}$
 are {\allowdisplaybreaks
\begin{align}\sum_{k=0}^\infty\binom{2k}k^3\frac{4k+1}{(-64)^k}&=\frac2{\pi},\label{eq:Rama1}
\\\sum_{k=0}^\infty\binom{2k}k^3\frac{6k+1}{256^k}&=\frac 4{\pi},\label{eq:Rama2}
\\\sum_{k=0}^\infty\binom{2k}k^3\frac{6k+1}{(-512)^k}&=\frac{2\sqrt2}{\pi},
\\\sum_{k=0}^\infty\binom{2k}k^3\frac{42k+5}{4096^k}&=\frac{16}{\pi}.\label{eq:Rama4}
\end{align}}Note that \eqref{eq:Rama1} was first obtained by Bauer \cite{Bauer1859}, and \eqref{eq:Rama2}--\eqref{eq:Rama4} were discovered by Ramanujan \cite{Ramanujan1914}. Usually, 
one proves such identities using the theory of modular forms (see, for example, \cite{HoeijTsaiYe2024,RogersStraub2013,WangYang2022}).
The identities \eqref{eq:Rama1}--\eqref{eq:Rama4} have the following variants involving second-order harmonic numbers $ \mathsf H^{(2)}_k\colonequals \sum_{0<n\leq k}\frac{1}{n^2}$:
{\allowdisplaybreaks\begin{align}\label{eq:-64H2}
\sum_{k=0}^\infty\binom{2k}k^3\left[\mathsf H_{2k}^{(2)}-\frac12\mathsf H_k^{(2)}\right]\frac{4k+1}{(-64)^k}&=-\frac{\pi}{12},
\\\label{eq:256H2}\sum_{k=0}^\infty\binom{2k}k^3\left[\mathsf H_{2k}^{(2)}-\frac5{16}\mathsf H_k^{(2)}\right]\frac{6k+1}{256^k}&=\frac{\pi}{12},
\\\label{eq:-512H2}\sum_{k=0}^\infty\binom{2k}k^3\left[\mathsf H_{2k}^{(2)}-\frac5{16}\mathsf H_k^{(2)}\right]
\frac{6k+1}{(-512)^k}&=-\frac{\sqrt{2}\pi}{48},
\\\label{eq:4096H2}\sum_{k=0}^\infty\binom{2k}k^3\left[\mathsf H_{2k}^{(2)}-\frac{25}{92}\mathsf H_k^{(2)}\right]\frac{42k+5}{4096^k}&
=\frac{2\pi}{69}.
\end{align}}Here is a brief history of the last four displayed formulae:
Wei and Ruan    \cite{WeiRuan2024} established \eqref{eq:-64H2};
Guo and Lian
\cite{GuoLian2021}  had conjectured \eqref{eq:256H2} and \eqref{eq:-512H2}  before both formulae were confirmed by Wei \cite{Wei2023}; Sun \cite[(107)]{Sun2022}  had proposed \eqref{eq:4096H2} before it was  verified by Wei \cite{Wei2023c}.

Motivated by the identities \eqref{eq:-64H2}--\eqref{eq:4096H2}, the first-named author has
recently
conjectured  the following four identities (cf.\ \cite{MOpost}):
{\allowdisplaybreaks\begin{align}\sum_{k=0}^\infty\binom{2k}k^3
\left[\mathsf H_{2k}^{(2)}-\frac1{2}\mathsf H_k^{(2)}+2L_{-8}(2)-\frac{5\pi^2}{24}\right]\frac{1}{(-64)^k}&=0,\label{eq:Sun1}
\\\sum_{k=0}^\infty\binom{2k}k^3\left[\mathsf H_{2k}^{(2)}-\frac5{16}\mathsf H_k^{(2)}
+\frac{135L_{-3}(2)-11\pi^2}{96}\right]\frac{1}{256^k}&=0,\label{eq:Sun2}
\\\sum_{k=0}^\infty \binom{2k}k^3\left[\mathsf H_{2k}^{(2)}-\frac5{16}\mathsf H_k^{(2)}
+\frac{120L_{-4}(2)-11\pi^2}{96}\right]\frac{1}{(-512)^k}&=0,
\\\sum_{k=0}^\infty\binom{2k}k^3\left[\mathsf H_{2k}^{(2)}-\frac{25}{92}\mathsf H_k^{(2)}+\frac{735L_{-7}(2)
-86\pi^2}{1104}\right]\frac{1}{4096^k}&=0.\label{eq:Sun4}\end{align}}Here, the Dirichlet $L$-function\begin{align}
L_d(s)\colonequals \sum_{n=1}^\infty \left(\frac{d}{n}\right)\frac{1}{n^s} ,\quad \RE s>1
\end{align}is defined via the Kronecker symbol $\big(\frac d{\cdot}\big)$.
 The Riemann zeta function\begin{align}
\zeta(s)\colonequals \sum_{n=1}^\infty\frac{1}{n^s},\quad \RE s>1\label{eq:zeta_defn}
\end{align}is the same as $ L_1(s)$.

To attack \eqref{eq:Sun1}--\eqref{eq:Sun4} and to revisit \eqref{eq:-64H2}--\eqref{eq:4096H2}  in \S\ref{sec:H2}, we give modular pa\-ra\-me\-tri\-za\-tions for our series of interest. We will need the Dedekind eta function $ \eta(z)\colonequals e^{\pi iz/12}\prod_{n=1}^\infty(1-e^{2\pi inz})$ for $z\in\mathfrak H\colonequals \{w\in\mathbb C|\I w>0\}$, from which we may define  the modular lambda function\begin{align}
\lambda(2z)\equiv\alpha_{4}(z)\colonequals\frac{2^{4}[\eta(z)]^8[\eta(4z)]^{16}}{[\eta(2z)]^{24}}=1-\alpha_4\left( -\frac{1}{4z} \right)=1-\lambda\left( -\frac{1}{2z} \right)\label{eq:lambda_defn}
\end{align}  and  the weight-4 holomorphic Eisenstein series \begin{align}
E_{4}(z)\colonequals {}&\frac{[\eta(z)]^{40}\{1-\lambda(z)+[\lambda(z)]^2\}}{[\eta(2z)\eta(z/2)]^{16}}=\frac{1}{2\zeta(4)}\sum_{\substack{m,n\in\mathbb Z\\m^2+n^2\neq0}}\frac{1}{(mz+n)^{4}}=1+240\nsum\limits_{n=1}^\infty\dfrac{n^3e^{2\pi inz}}{1-e^{2\pi inz}}
\label{eq:E4defn}
\end{align} for $ z\in\mathfrak H$.
In addition, we will avail ourselves of the Eichler integral\begin{align}
\mathscr
E_4(z)\colonequals\int_{z}^{i\infty} [1-E_4(w)](w-z)^2\D w,\label{eq:EichlerE4defn}
\end{align}the real-analytic Eisenstein series (also known as the Epstein zeta function)\begin{align}
E(z,s)\colonequals \frac{1}{2\zeta(2s)}\sum_{\substack{m,n\in\mathbb Z\\m^2+n^2\neq0}}\frac{(\I z)^s}{|mz+n|^{2s}}=E(z_{_{}}+1,s)=E\left( -\frac{1}{z} ,s\right),\label{eq:EZF_defn}
\end{align}the weight-2 non-holomorphic Eisenstein series  \begin{align}
E_{2}(z)\colonequals\frac{12}{\pi i\eta(z)}\frac{\partial\eta(z)}{\partial z}-\frac{3}{\pi\I z}=1-{\frac{3}{\pi\I z}-24\sum_{n=1}^\infty\frac{ne^{2\pi inz}}{1-e^{2\pi inz}} } ,
\end{align} and the Legendre--Ramanujan function  (cf.\ \cite[(2.2.21)]{AGF_PartI})
\begin{align}
R_{-1/2}(1-2\alpha_4(z))={}&-\frac{2^{4}[E_{2}(4z)]^2-2^{4}E_{4}(4 z)-[E_{2}(z)]^2+E_{4}(z)}{2[4E_{2}(4z)-E_{2}(z)]^{2}}.
\end{align}Whenever a point $ z$ in the upper-half plane $ \mathfrak H$ is a quadratic irrational number (namely, the  minimal polynomial of $z$  has degree $2$), both $ \alpha_4(z)$ and $ R_{-1/2}(1-2\alpha_4(z))$ are explicitly computable algebraic numbers  \cite[\S6]{Zagier2008Mod123}.
\begin{theorem}\label{thm:H2}\begin{enumerate}[leftmargin=*,  label=\emph{(\alph*)},ref=(\alph*),
widest=d, align=left] \item
If  we have either $ 2z/i\geq1$, or $ \RE z=\frac12$ and $ \I z\geq\frac{1}{\sqrt{2}}$,  then {\allowdisplaybreaks
\begin{align}\begin{split}\mathscr Q_1(z)\colonequals {}&
\frac{\displaystyle \sum_{k=0}^\infty\binom{2k}{k}^3\left[ \mathsf H_{2k}^{(2)}-\frac{1}{4}\mathsf H_{ k}^{(2)} \right]\left\{\frac{\alpha_{4}(z)[1-\alpha_{4}(z)]}{2^{4}}\right\}^k}{\displaystyle  \sum_{k=0}^\infty\binom{2k}{k}^3\left\{\frac{\alpha_{4}(z)[1-\alpha_{4}(z)]}{2^{4}}\right\}^k}\\={}&\frac{7\zeta(3)}{4\pi\I z }-\frac{\pi^2\left[4 E \big(z+\frac12,2\big)-E(2z,2)\right]}{90}-\frac{\pi^2 i\left[8\mathscr  E_{4} \big(z+\frac12\big)-\mathscr  E_{4}(2z)\right]}{120\I z}
\end{split}\label{eq:Q1}\intertext{and}\begin{split}\mathscr Q_2(z)\colonequals {}&\frac{\displaystyle \sum_{k=0}^\infty\binom{2k}{k}^3\mathsf H_{ k}^{(2)}\left\{\frac{\alpha_{4}(z)[1-\alpha_{4}(z)]}{2^{4}}\right\}^k}{\displaystyle  \sum_{k=0}^\infty\binom{2k}{k}^3\left\{\frac{\alpha_{4}(z)[1-\alpha_{4}(z)]}{2^{4}}\right\}^k}\\={}&-\frac{2\pi^2}{3}(\I z)^2-\frac{2\zeta(3)}{\pi\I z }-\frac{2\pi^2\left[ E \big(z+\frac12,2\big)-4E(2z,2)\right]}{45}\\&{}-\frac{\pi^2 i\left[\mathscr  E_{4} \big(z+\frac12\big)-2\mathscr  E_{4}(2z)\right]}{15\I z}.
\end{split}\label{eq:Q2}
\end{align}
}
\item If  we have either $ 2z/i\geq1$, or $ \RE z=\frac12$ and $ \I z\geq\frac{1}{\sqrt{2}}$,  then{\allowdisplaybreaks
\begin{align}
\begin{split}&\mathscr R_1(z)\\\colonequals {}&\sum_{k=0}^\infty\binom{2k}{k}^3\left[ \mathsf H_{2k}^{(2)}-\frac{1}{4}\mathsf H_{ k}^{(2)} \right]\frac{2[1-2\alpha_4(z)]k+ R_{-1/2}(1-2\alpha_4(z))}{\I z}\left\{\frac{\alpha_{4}(z)[1-\alpha_{4}(z)]}{2^{4}}\right\}^k
\\={}&\frac{\mathscr Q_1(z)}{\pi(\I z)^2}-\frac{\pi i\left[2\mathscr  E_{4} ''\big(z+\frac12\big)-\mathscr  E_{4}''(2z)\right]}{30\I z}
\end{split}\label{eq:R1(z)}\intertext{and}\begin{split}&\mathscr R_2(z)\\\colonequals {}&\sum_{k=0}^\infty\binom{2k}{k}^3\mathsf H_{ k}^{(2)}\frac{2[1-2\alpha_4(z)]k+ R_{-1/2}(1-2\alpha_4(z))}{\I z}\left\{\frac{\alpha_{4}(z)[1-\alpha_{4}(z)]}{2^{4}}\right\}^k
\\={}&\frac{\mathscr Q_2(z)}{\pi(\I z)^2}-\frac{\pi i\left[\mathscr  E_{4} ''\big(z+\frac12\big)-8\mathscr  E_{4}''(2z)\right]}{15\I z},
\end{split}\label{eq:R2(z)}
\end{align}
}where\begin{align}\label{eq:EichlerE4''}
\mathscr E''_4(z)\colonequals \frac{\partial^2\mathscr
E_4(z)}{\partial z^{2}}=2\int_{z}^{i\infty} [1-E_4(w)]\D w.
\end{align}
\end{enumerate}
\end{theorem}
\begin{remark}For many quadratic irrational numbers    $z$, Glasser and Zucker \cite{GlasserZucker1980} have produced closed-form evaluations of $ E(z,s)$ in terms of Dirichlet $L$-functions $ L_d(s)$, with suitably chosen $d$. All the special values of Epstein zeta functions in Table \ref{tab:H2} can be read off from \cite[Table VI]{GlasserZucker1980}.
\eor\end{remark}\begin{remark}When $z$ is a quadratic irrational number listed in the first column of Table \ref{tab:H2}, there is a corresponding rational number $ r=r(z)\in\mathbb Q$, such that the sum\begin{align}
\begin{split}\mathscr S_{r}
(z)\colonequals {}&\frac{7\zeta(3)}{4\pi\I z }-\frac{\pi^2 i\left[8\mathscr  E_{4} \big(z+\frac12\big)-\mathscr  E_{4}(2z)\right]}{120\I z}\\{}&+r\left\{ \frac{2\pi^2}{3}(\I z)^2+\frac{2\zeta(3)}{\pi\I z } +\frac{\pi^2 i\left[\mathscr  E_{4} \big(z+\frac12\big)-2\mathscr  E_{4}(2z)\right]}{15\I z}\right\}\label{eq:Sr(z)}
\end{split}\end{align}becomes a  rational multiple of $ \pi^2$, and that \begin{align}\mathscr T_r(z)\colonequals\mathscr  R_1(z)-r\mathscr R_2(z)\label{eq:Tr(z)}
\end{align} becomes a  rational multiple of $ \pi$. (See \S\ref{subsec:sumEichler} for the explicit computations.)

The aforementioned property of $ \mathscr S_{r(z)}(z)$, together with the reductions of Epstein zeta functions to Dirichlet $L$-values, allow us to prove \eqref{eq:Sun1}--\eqref{eq:Sun4} (cf.\ the penultimate column of Table \ref{tab:H2}), which have been recently conjectured by the first-named author in Question 507709 of  MathOverflow \cite{MOpost}. Note that Eric Naslund has posted an answer to  \cite{MOpost} that contained a complete proof of \eqref{eq:Sun1}, but the more challenging identities \eqref{eq:Sun2}--\eqref{eq:Sun4} did not succumb to his method.

The corresponding property of $ \mathscr T_{r(z)}(z)$ (cf.\ the last column of Table \ref{tab:H2}) constitutes an alternative way to interpret  \eqref{eq:-64H2}--\eqref{eq:4096H2}. \eor\end{remark}

\begin{table}[h]\caption{\label{tab:H2}Selected arithmetic data for special cases of Theorem \ref{thm:H2}}{\small\begin{align*}\begin{array}{c|l@{\,\,}l@{\,}l|l@{\;\,}l|l@{\;\,}l@{\;\,}l}\hline\hline
\vphantom{\frac\int\int}z & \frac{\alpha_{4}(z)[1-\alpha_{4}(z)]}{2^4}&\frac{1-2\alpha_4(z)}{\I z}&\frac{ R_{-1/2}(1-2\alpha_4(z))}{2[1-2\alpha_4(z)]}&E\big(z+\frac12,2\big)&E(2z,2)& r(z) & \mathscr Q_1(z)-r(z)\mathscr Q_2(z)&\mathscr T_{r(z)}(z) \\\hline
\vphantom{\frac\int\int}\frac{\sqrt{3}i}{2} & \frac{1}{2^{8}}&1&\frac{1}{6}&\frac{135L_{-3}(2)}{4\pi^2}&\frac{405L_{-3}(2)}{8\pi^2} & \frac{1}{16} & \frac{11\pi^{2}}{96}-\frac{45L_{-3}(2)}{32}&\frac{\pi}{36} \\
\vphantom{\frac\int\int}\frac{\sqrt{7}i}{2} & \frac{1}{2^{12}} &\frac{3}{4}&\frac{5}{42}&\frac{105L_{-7}(2)}{4\pi^2}&\frac{525L_{-7}(2)}{8\pi^2}& \frac{1}{46} & \frac{43\pi^{2}}{552}-\frac{245L_{-7}(2)}{368} &\frac{\pi}{966}\\\hline
\vphantom{\frac\int\int}\frac{1}{2}+\frac{i}{\sqrt{2}} & -\frac{1}{2^{6}}&2&\frac14&\frac{30L_{-8}(2)}{\pi^2}&\frac{30L_{-8}(2)}{\pi^2} & \frac{1}{4} & \frac{5\pi^{2}}{24}-2L_{-8}(2) &-\frac{\pi}{12}\\
\vphantom{\frac\int\int}\frac{1}{2} +i& -\frac{1}{2^{9}} &\frac{3}{2\sqrt{2}}&\frac16&\frac{30L_{-4}(2)}{\pi^2}&\frac{105L_{-4}(2)}{2\pi^2}& \frac{1}{16} & \frac{11\pi^{2}}{96} -\frac{5L_{-4}(2)}{4}&-\frac{\pi}{96}\\\hline\hline
\end{array}\end{align*}}\end{table}

Inspired by the classical Ramanujan-type series  \eqref{eq:Rama1}--\eqref{eq:Rama4} as well as their variants \eqref{eq:-64H2}--\eqref{eq:4096H2}, the first-named author discovered the following identities \cite[(90), (99), (104), (110)]{Sun2022}  experimentally:{\allowdisplaybreaks
\begin{align}\sum_{k=0}^\infty\binom{2k}k^3\mathsf H^{(3)}_{2k}\frac{4k+1}{(-64)^k}&=\frac{15\zeta(3)}{4\pi}-2L_{-4}(2),\label{eq:SunH3a}
\\\sum_{k=0}^\infty\binom{2k}k^3\left[\mathsf H_{2k}^{(3)}-\frac{7}{64}\mathsf H_k^{(3)}\right]\frac{6k+1}{256^k}&= \frac{25\zeta(3)}{8\pi}-L_{-4}(2),\label{eq:SunH3b}
\\\sum_{k=0}^\infty\binom{2k}k^3\left[\mathsf H_{2k}^{(3)}-\frac{7}{64}\mathsf H_k^{(3)}\right]\frac{6k+1}{(-512)^k}&=\frac{57\zeta(3)}{16\sqrt{2}\pi}-L_{-8}(2),
\label{eq:SunH3c}\\\sum_{k=0}^\infty\binom{2k}k^3\left[\mathsf H_{2k}^{(3)}-\frac{43}{352}\mathsf H_k^{(3)}\right]\frac{42k+5}{4096^k}&=\frac{555\zeta(3)}{77\pi}-\frac{32L_{-4}(2)}{11},\label{eq:SunH3d}
\end{align}}where  $ \mathsf H^{(3)}_k\colonequals \sum_{0<n\leq k}\frac{1}{n^3}$.
Some of these formulae were subsequently proved: Wei and Xu \cite{WeiXu2023} verified \eqref{eq:SunH3a}; Li and Chu \cite{LiChu2025Axioms} confirmed \eqref{eq:SunH3b}; Li and Chu \cite{LiChu2024} demonstrated \eqref{eq:SunH3c}.
The last one
\eqref{eq:SunH3d} has remained hitherto open.

In \S\ref{sec:H3}, we will unify \eqref{eq:SunH3a}--\eqref{eq:SunH3d} in a modular framework. On top of the prerequisites for Theorem \ref{thm:H2}, we will use the weight-6 holomorphic Eisenstein series \begin{align}
\begin{split}E_{6}(z)\colonequals {}&\frac{[\eta(z)]^{60}[1+\lambda(z)][2-\lambda(z)][1-2\lambda(z)]}{2[\eta(2z)\eta(z/2)]^{24}}\\={}&\frac{1}{2\zeta(6)}\sum_{\substack{m,n\in\mathbb Z\\m^2+n^2\neq0}}\frac{1}{(mz+n)^{6}}=1-504\nsum\limits_{n=1}^\infty\dfrac{n^5e^{2\pi inz}}{1-e^{2\pi inz}}
\end{split}\label{eq:E6defn}
\end{align} for $ z\in\mathfrak H$, along with an associated Eichler integral\begin{align}
\mathscr E_6(z)\colonequals\int_{z}^{i\infty} [1-E_6(w)](z-w)^{4}\D w. \label{eq:EichlerE6defn}
\end{align}
\begin{theorem}\begin{enumerate}[leftmargin=*,  label=\emph{(\alph*)},ref=(\alph*),
widest=d, align=left] \item
If  we have either $ 2z/i\geq1$, or $ \RE z=\frac12$ and $ \I z\geq\frac{1}{\sqrt{2}}$,  then \begin{align}
\frac{\displaystyle \sum_{k=0}^\infty\binom{2k}{k}^3\left[ \mathsf H_{2k}^{(3)}-\frac{1}{8}\mathsf H_{ k}^{(3)} \right]\left\{\frac{\alpha_{4}(z)[1-\alpha_{4}(z)]}{2^{4}}\right\}^k}{\displaystyle  \sum_{k=0}^\infty\binom{2k}{k}^3\left\{\frac{\alpha_{4}(z)[1-\alpha_{4}(z)]}{2^{4}}\right\}^k}={}&\frac{\pi ^3 i\big[\mathscr E''_6(2z)-8\mathscr E''_6\big(z+\frac12\big)\big]}{1512}\label{eq:H3Q1}\intertext{and}
\begin{split}\frac{\displaystyle \sum_{k=0}^\infty\binom{2k}{k}^3\mathsf H_{ k}^{(3)}\left\{\frac{\alpha_{4}(z)[1-\alpha_{4}(z)]}{2^{4}}\right\}^k}{\displaystyle  \sum_{k=0}^\infty\binom{2k}{k}^3\left\{\frac{\alpha_{4}(z)[1-\alpha_{4}(z)]}{2^{4}}\right\}^k}={}&\frac{\pi ^3 i\big[\mathscr E''_6\big(z+\frac12\big)-8\mathscr E''_6(2z)\big]}{189}\\&{}-\frac{\pi^3 i\big[4\mathscr E_4(z)-\mathscr E_4(4z)\big]}{15},\end{split}\label{eq:H3Q2}
\end{align}
where\begin{align}
\mathscr E''_6(z)\colonequals \frac{\partial^2\mathscr
E_6(z)}{\partial z^{2}}=12\int_{z}^{i\infty} [1-E_6(w)](z-w)^{2}\D w.
\end{align}
\item If  we have either $ 2z/i\geq1$, or $ \RE z=\frac12$ and $ \I z\geq\frac{1}{\sqrt{2}}$,  then \begin{align}
\begin{split}&\sum_{k=0}^\infty\binom{2k}{k}^3\left[ \mathsf H_{2k}^{(3)}-\frac{1}{8}\mathsf H_{ k}^{(3)} \right]\frac{2[1-2\alpha_4(z)]k+ R_{-1/2}(1-2\alpha_4(z))}{\I z}\left\{\frac{\alpha_{4}(z)[1-\alpha_{4}(z)]}{2^{4}}\right\}^k\\={}&\frac{\pi ^2 i\big[\mathscr E''_6(2z)-8\mathscr E''_6\big(z+\frac12\big)\big]}{1512(\I z)^{2}}+\frac{\pi ^2 \big[\mathscr E'''_6(2z)-4\mathscr E'''_6\big(z+\frac12\big)\big]}{756\I z}\end{split}\intertext{and}\begin{split}&\sum_{k=0}^\infty\binom{2k}{k}^3\mathsf H_{ k}^{(3)}\frac{2[1-2\alpha_4(z)]k+ R_{-1/2}(1-2\alpha_4(z))}{\I z}\left\{\frac{\alpha_{4}(z)[1-\alpha_{4}(z)]}{2^{4}}\right\}^k\\={}&\frac{8\pi^2\I z}{3}-\frac{6\zeta(3)}{\pi(\I z)^{2}}-\frac{8\pi^{2}[E(4z,2)-E(z,2)]}{45\I z}\\&{}+\frac{\pi ^2 i\big[\mathscr E''_6\big(z+\frac12\big)-8\mathscr E''_6(2z)\big]}{189(\I z)^{2}}+\frac{\pi ^2 \big[\mathscr E'''_6\big(z+\frac12\big)-16\mathscr E'''_6(2z)\big]}{189(\I z)^{2}},
\end{split}
\end{align}where\begin{align}
\mathscr E'''_6(z)\colonequals \frac{\partial^3\mathscr
E_6(z)}{\partial z^{3}}=24\int_{z}^{i\infty} [1-E_6(w)](z-w)^{}\D w.
\end{align}
\end{enumerate}\label{thm:H3}
\end{theorem}
\begin{remark} To fill in the second column of Table \ref{tab:H3}, one may consult  \cite[Table VI]{GlasserZucker1980}. For instance, the entry \cite[Table VI, $ (a,b,c)=(1,0,28)$]{GlasserZucker1980} can be transcribed as follows:\begin{align}
E\big(2\sqrt{7}i,s\big)=\frac{\big(2\sqrt{7}\big)^{s}}{2\zeta(2s)}\left[ \left(1-\frac{1}{2^{s-1}}+\frac{3}{2^{2 s}}-\frac{1}{2^{3 s-2}}+\frac{1}{2^{4 s-2}}\right) L_{1}(s)L_{-7}(s)+L_{-4}(s)L_{28}(s)\right],
\end{align}  while we have $ L_1(2)=\zeta(2)=\frac{\pi^2}{6}$ and $ L_{28}(2)=\frac{2 \pi ^2}{7 \sqrt{7}}$. Equipped with the second column of Table \ref{tab:H3} and some entries of Table \ref{tab:H2}, one may compute [cf.\ \eqref{eq:E(z,2)add} below]\begin{align}
E(4z,2)-E(z,2)=E\left( z+\frac{1}{2} ,2\right)-\frac{9}{2}E(2z,2)+2E(4z,2),
\end{align}  as shown in the third column of Table \ref{tab:H3}. Note that   these special values of   $E(4z,2)-E(z,2) $ have fully accounted for all the occurrences of $ L_d(2)$  on the right-hand sides of \eqref{eq:SunH3a}--\eqref{eq:SunH3d}.
\eor\end{remark}\begin{remark}When $z$ is a quadratic irrational number listed in the first column of Table \ref{tab:H3}, there is a corresponding rational number $ \check r=\check r(z)\in\mathbb Q$, such that the sum\begin{align}
\begin{split}\check{\mathscr T}_{\check r}
(z)\colonequals {}&\frac{\pi ^2 i\big[\mathscr E''_6(2z)-8\mathscr E''_6\big(z+\frac12\big)\big]}{1512(\I z)^{2}}+\frac{\pi ^2 \big[\mathscr E'''_6(2z)-4\mathscr E'''_6\big(z+\frac12\big)\big]}{756\I z}+\check r\left\{\frac{8\pi^2\I z}{3}-\frac{6\zeta(3)}{\pi(\I z)^{2}} \right.\\&\left.{}+\frac{\pi ^2 i\big[\mathscr E''_6\big(z+\frac12\big)-8\mathscr E''_6(2z)\big]}{189(\I z)^{2}}+\frac{\pi ^2 \big[\mathscr E'''_6\big(z+\frac12\big)-16\mathscr E'''_6(2z)\big]}{189(\I z)^{2}}\right\}\label{eq:Ur(z)}
\end{split}\end{align}becomes a  rational multiple of $ \frac{\zeta(3)}{\pi}$.   \big[See \S\ref{subsec:specEichlerE6} for the detailed computations that draw on selected values of  $E\big(z+\frac12,3\big)$ and $E(2z,3)$.\big]

For each listed quadratic irrational $ z$, the entry $ E\big(z+\frac12,2\big)$ [resp.\ $ E(2z,2)$] in Table \ref{tab:H2} and the entry $ E\big(z+\frac12,3\big)$ [resp.\ $ E(2z,3)$] in Table \ref{tab:H3} appear in drastically different forms, but they are indeed special cases of the same formula involving Dirichlet  $L$-functions. For example, the identity  \cite[Table VI, $ (a,b,c)=(1,0,7)$]{GlasserZucker1980}\begin{align}
E\big(\sqrt{7}i,s\big)=\frac{\big(\sqrt{7}\big)^{s}}{\zeta(2s)} \left(1-\frac{1}{2^{s-1}}+\frac{1}{2^{2 s-1}}\right)L_{1}(s)L_{-7}(s)
\end{align}entails the exact values
 of both $ E\big(\sqrt{7}i,2\big)$ and $ E\big(\sqrt{7}i,3\big)$.
\eor\end{remark}\begin{remark}In addition to \eqref{eq:SunH3d}, the first-named author also proposed the following formula (cf.\ \cite[(5.13)]{Sun2026binomII}): \begin{align}
\sum_{k=0}^\infty\binom{2k}k^3\left[(42k+5)\mathsf H_k^{(3)}-\frac{352}{(2k+1)^2}\right]\frac{1}{4096^k}&=\frac{32}{7}\left[ \frac{335\zeta(3)}{\pi} -224L_{-4}(2)\right],\label{eq:SunH3e}
\end{align}based on numerical evidence. One can  now prove \eqref{eq:SunH3e}  by eliminating the $ \mathsf H_{2k}^{(3)}$ term from a linear combination of  \eqref{eq:SunH3d} [to be demonstrated in Proposition \ref{prop:H3proof}(b)] and an identity established by Wei and Xu (cf.\ \cite[(1.9)]{WeiXu2025}):
\begin{align}
\sum_{k=0}^\infty\binom{2k}k^3\left\{(42k+5)\left[17\mathsf H^{(3)}_{2k}-2\mathsf H_k^{(3)}\right]-\frac{27}{(2k+1)^2}\right\}\frac{1}{4096^k}=\frac{240\zeta(3)}{\pi} -128L_{-4}(2).\label{eq:WeiXu}
\end{align}In \cite[\S4]{WeiXu2025}, Wei and Xu verified \eqref{eq:WeiXu} by hypergeometric transformations, extending Wei's hypergeometric approach \cite{Wei2023c} to \eqref{eq:4096H2}. However, such a hypergeometric method was not sharp enough to settle either \eqref{eq:SunH3d} or \eqref{eq:SunH3e} individually.
\eor\end{remark}

\begin{table}[h]\caption{\label{tab:H3}Selected arithmetic data for special cases of Theorem \ref{thm:H3}}{\small\begin{align*}\begin{array}{c|ll|ll|ll}\hline\hline
\vphantom{\frac\int\int}z & E(4z,2)& E(4z,2)-E(z,2)&E\big(z+\frac12,3\big)&E(2z,3)& \check r(z)&\check{\mathscr T}_{\check r(z)}(z) \\\hline
\vphantom{\frac\int\int}\frac{\sqrt{3}i}{2} &\frac{3105L_{-3}(2)}{32\pi^{2}}+\frac{30\sqrt{3}L_{-4}(2)}{\pi^{2}}&\frac{60\sqrt{3}L_{-4}(2)}{\pi^{2}} &\frac{105\zeta (3)}{2 \pi ^3} & \frac{1155 \zeta (3)}{8 \pi ^3} &\frac1{64}&\frac{25\zeta(3)}{24\pi}\\
\vphantom{\frac\int\int}\frac{\sqrt{7}i}{2} & \frac{4305L_{-7}(2)}{32\pi^{2}}+\frac{360L_{-4}(2)}{\sqrt{7}\pi^2}&\frac{720L_{-4}(2)}{\sqrt{7}\pi^{2}}&\frac{540 \zeta (3)}{7 \pi ^3}&\frac{3375 \zeta (3)}{7 \pi ^3} &\frac1{352}& \frac{555 \zeta (3)}{2156 \pi } \\\hline
\vphantom{\frac\int\int}\frac{1}{2}+\frac{i}{\sqrt{2}} &\frac{105L_{-8}(2)}{2\pi^2}+\frac{45L_{-4}(2)}{\sqrt{2}\pi^{2}}&\frac{45\sqrt{2}L_{-4}(2)}{\pi^{2}} &\frac{2835 \zeta (3)}{32 \pi ^3} &\frac{2835 \zeta (3)}{32 \pi ^3} &\frac18&\frac{15\zeta(3)}{4\pi} \\
\vphantom{\frac\int\int}\frac{1}{2} +i&\frac{825L_{-4}(2)}{8\pi^2}+\frac{45\sqrt{2}L_{-8}(2)}{\pi^2} &\frac{90\sqrt{2}L_{-8}(2)}{\pi^2} &\frac{945 \zeta (3)}{16 \pi ^3}&\frac{27405 \zeta (3)}{128 \pi ^3}&\frac1{64}&\frac{57 \zeta (3)}{64 \pi } \\\hline\hline
\end{array}\end{align*}}\end{table}

In \S\ref{sec:outlook}, we wrap up this article with discussions of some problems related to Theorem \ref{thm:H2} and Theorem  \ref{thm:H3}.\section{\protect{Infinite series involving $ \binom{2k}k^3$, $ \mathsf H_k^{(2)}$, and $ \mathsf H_{2k}^{(2)}$}\label{sec:H2}}In \S\ref{subsec:pFqODE}, we construct integral representations for our series of interest, by solving inhomogeneous differential equations. In \S\ref{subsec:modH2}, we introduce modular parametrizations for the series and their integral representations, thereby proving the main statements of Theorem \ref{thm:H2}. In \S\ref{subsec:sumEichler}, we study some special cases  of Theorem \ref{thm:H2}, as listed in  Table \ref{tab:H2}.

 \subsection{Second-order hypergeometric deformations and variation of parameters\label{subsec:pFqODE}}Hereafter, we will use the following notation for a convergent generalized hypergeometric series
\begin{align}{_pF_q}\left(\left.\begin{array}{@{}c@{}}
a_{1},\dots,a_p \\[4pt]
b_{1},\dots,b_q \\
\end{array}\right| t\right):=1+\sum_{n=1}^\infty\frac{\prod_{j=1}^p(a_{j})_n}{\prod_{k=1}^q(b_{k})_n }\frac{t^n}{n!},\label{eq:defn_pFq}\end{align}where $ (a)_{n}\colonequals\prod_{m=0}^{n-1}(a+m)$ is the rising factorial, and $ b_1,\dots,b_q$ are neither zero nor negative integers. Occasionally, we will also write $ _pF_q\left(\!\begin{smallmatrix}a_1,\dots,a_p\\b_1,\dots,b_q\end{smallmatrix}\!\middle| t\right)$ for the analytic continuation of the convergent series defined above.
For example, the complete elliptic integral of the first kind\begin{align}
\mathbf K\big(\sqrt{t}\big)\colonequals \int_0^{\pi/2}\frac{\D \theta}{\sqrt{1-t\sin^2\theta}}=\frac{\pi}{2}{_2F_1}\left(\left.\begin{array}{@{}c@{}}
\frac{1}{2},\frac{1}{2} \\[4pt]
1 \\
\end{array}\right| t\right)
\end{align}is defined for $ t\in\mathbb C\smallsetminus[1,\infty)$, whereas the corresponding hypergeometric series converges absolutely for $ |t|<1$.

When we talk about a ``hypergeometric deformation'', we refer to every possible  (first-order or higher-order) partial derivative of a generalized hypergeometric series  $ _pF_q\left(\!\begin{smallmatrix}a_1,\dots,a_p\\b_1,\dots,b_q\end{smallmatrix}\!\middle| t\right)$
with respect to its hypergeometric parameters $a_1,\dots,a_p$ and $b_1,\dots,b_q $.  These hypergeometric deformations usually satisfy some inhomogeneous differential equations, which in turn, can be solved by a standard method called ``variation of parameters''. Such solutions lead to integral representations of certain hypergeometric deformations of our interest, as illustrated by the lemma below.\begin{lemma}\label{lm:var_para}For $ |4t(1-t)|\leq1$ and $\RE t\leq\frac12 $, we have the following integral representations for hypergeometric deformations:{\allowdisplaybreaks\begin{align}
\begin{split}{}&\sum_{k=0}^\infty\binom{2k}{k}^3\left[ \mathsf H_{2k}^{(2)}-\frac{1}{4}\mathsf H_{ k}^{(2)} \right]\left[\frac{t(1-t)}{2^{4}}\right]^k=-\frac{1}{8}\left. \!\frac{\partial^2}{\partial\nu^2} {_3}F_2\left(\left. \begin{array}{@{}c@{}}
-\nu,1+\nu,\frac{1}{2}\ \\1,1\ \\
\end{array} \right| 4t(1-t)\right)\right|_{\nu=-1/2}\\={}&\left(\frac{2}{\pi}\right)^{3}\int_0^t\mathbf K\big(\sqrt{\mathstrut s}\big)\mathbf K\big(\sqrt{\mathstrut t}\big)\big[\mathbf K\big(\sqrt{\mathstrut 1-s}\big)\mathbf K\big(\sqrt{\mathstrut t}\big)-\mathbf K\big(\sqrt{\mathstrut s}\big)\mathbf K\big(\sqrt{\mathstrut 1-t}\big)\big]\D s,\end{split}
\label{eq:nu''}\\\begin{split}{}&\sum_{k=0}^\infty\binom{2k}{k}^3\mathsf H_{ k}^{(2)}\left[\frac{t(1-t)}{2^{4}}\right]^k=\frac{1}{2}\left. \!\frac{\partial^2}{\partial\varepsilon^2} {_3}F_2\left(\left. \begin{array}{@{}c@{}}
\frac{1}{2},\frac{1}{2},\frac{1}{2}\ \\1-\varepsilon,1+\varepsilon\ \\
\end{array} \right| 4t(1-t)\right)\right|_{\varepsilon=0}
\\={}&\left(\frac{2}{\pi}\right)^{3}\int_{1/2}^t\mathbf K\big(\sqrt{\mathstrut s}\big)\mathbf K\big(\sqrt{\mathstrut t}\big)\big[\mathbf K\big(\sqrt{\mathstrut 1-s}\big)\mathbf K\big(\sqrt{\mathstrut t}\big)-\mathbf K\big(\sqrt{\mathstrut s}\big)\mathbf K\big(\sqrt{\mathstrut 1-t}\big)\big]\frac{\D s}{s(1-s)}\\{}&-\frac{\big[\mathbf K\big(\sqrt{\mathstrut t}\big)\big]^{2}}{3}-\big[\mathbf K\big(\sqrt{\mathstrut 1-t}\big)\big]^{2}+\frac{2^{4}\mathbf K\big(\sqrt{\mathstrut t}\big)\mathbf K\big(\sqrt{\mathstrut 1-t}\big)G}{\pi^{2}},\end{split}\label{eq:eps''}
\end{align}}where $ G=L_{-4}(2)\colonequals \sum_{n=0}^\infty\frac{(-1)^n}{(2n+1)^2}$ is Catalan's constant.
\end{lemma}\begin{proof}The first equalities in \eqref{eq:nu''} and \eqref{eq:eps''} are both consequences of termwise derivatives of the  $_3F_2$ series.

Before proceeding, we recall  from \cite[(3.21)]{Zhou2025NotesBHS} that  an identity of  Clausen's type \begin{align}
{_3}F_2\left(\left. \begin{array}{@{}c@{}}
-\nu,1+\nu,\frac{1}{2}\ \\1+\varepsilon,1-\varepsilon\ \\
\end{array} \right| 4t(1-t)\right)=\frac{\varepsilon\pi}{\sin(\varepsilon\pi)}P^{-\varepsilon}_\nu(1-2t)P^{\varepsilon}_\nu(1-2t)
\label{eq:Clausen'}\end{align}holds for  $ \varepsilon\notin\mathbb Z$,  $ |4t(1-t)|\leq1$ and $\RE t\leq \frac12 $, where (cf.\ \cite[(3.20)]{Zhou2025NotesBHS})\footnote{Throughout this article,  Euler's gamma function $ \Gamma(s)$ for $s\in\mathbb C\smallsetminus\mathbb Z_{\leq0}$ is defined by  analytically  continuing $ \Gamma(s)\colonequals \int_0^\infty t^{s-1}e^{-t}\D t$ for $ \RE s>0$. 
}\begin{align}
\begin{split}P^{-\varepsilon}_\nu(1-2t)\colonequals \frac{1}{\Gamma(1+\varepsilon)}{}&{_2}F_1\left(\left. \begin{array}{@{}c@{}}
-\nu,1+\nu\ \\
1+\varepsilon\ \\
\end{array} \right| t\right)\left(\frac{t}{1-t}\right)^{\varepsilon/2}
\end{split}\label{eq:Pnumu}
\end{align} is the associated Legendre function of degree $\nu$ and order $-\varepsilon$. In particular, $ P_\nu^{}\equiv P_\nu^0$ is the Legendre function of degree $\nu$, and $ P_{-1/2}(1-2t)=\frac{2}{\pi}\mathbf K\big(\sqrt{t}\big)$.

In view of the last paragraph, we may write the infinite series in  \eqref{eq:nu''} as \begin{align}
-\frac{P_{-1/2}(1-2t)}{4}\left.\!\frac{\partial^2P_\nu(1-2t)}{\partial\nu^2}\right|_{\nu=-1/2}=-\frac{\mathbf K\big(\sqrt{t}\big)}{2\pi}\left.\!\frac{\partial^2P_\nu(1-2t)}{\partial\nu^2}\right|_{\nu=-1/2}.
\end{align}   Here,  twice differentiating the Legendre differential equation \begin{align}
\widehat{\mathscr L}_{\nu}
 P_\nu(1-2t)\colonequals \frac{\partial }{\partial t}\left[ t(1-t)\frac{\partial P_\nu(1-2t)}{\partial t} \right]+\nu(1+\nu)P_\nu(1-2t)=0
\end{align}with respect to $\nu$ at $ \nu=-\frac12$, we get\begin{align}
\widehat{\mathscr L}_{-1/2}\left.\!\frac{\partial^2P_\nu(1-2t)}{\partial\nu^2}\right|_{\nu=-1/2}=-2P_{-1/2}(1-2t).
\end{align}Such an inhomogeneous Legendre differential equation for a second-order hypergeometric deformation [equipped with a natural boundary condition\begin{align}
\left.\!\frac{\partial^2P_\nu(1)}{\partial\nu^2}\right|_{\nu=-1/2}=0
\end{align} that descends from the identity $ P_\nu(1)\equiv1$] can be solved by variation of parameters. The corresponding result \begin{align}
\left.\!\frac{\partial^2P_\nu(1-2t)}{\partial\nu^2}\right|_{\nu=-1/2}=-\frac{16}{\pi^{2}}\int_0^t\mathbf K\big(\sqrt{\mathstrut s}\big)\big[\mathbf K\big(\sqrt{\mathstrut 1-s}\big)\mathbf K\big(\sqrt{\mathstrut t}\big)-\mathbf K\big(\sqrt{\mathstrut s}\big)\mathbf K\big(\sqrt{\mathstrut 1-t}\big)\big]\D s
\end{align}has  been essentially reported in \cite[Theorem 4.2 and Lemma 4.4]{Zhou2025NotesBHS}. This explains  \eqref{eq:nu''}  in full.

Before deducing the integral representation in \eqref{eq:eps''}, we rewrite the infinite series in question as\begin{align}
\frac{2\big[\mathbf K\big(\sqrt{t}\big)\big]^{2}}{3}-\left[\left.\!\frac{\partial P_{-1/2}^\varepsilon(1-2t)}{\partial\varepsilon}\right|_{\varepsilon=0} \right]^2+\frac{2\mathbf K\big(\sqrt{t}\big)}{\pi}\left.\!\frac{\partial^2 P_{-1/2}^{\varepsilon}(1-2t)}{\partial\varepsilon^2}\right|_{\varepsilon=0}.
\end{align}We may transcribe \cite[(2.15)]{Zhou2025NotesBHS} into\begin{align}
\left.\!\frac{\partial P_{-1/2}^\varepsilon(1-2t)}{\partial\varepsilon}\right|_{\varepsilon=0}=\mathbf K\big(\sqrt{1-t}\big)-\frac{2\mathbf K\big(\sqrt{t}\big)}{\pi}(\gamma_0+2\log2),\label{eq:eps'}
\end{align}  where $\gamma_0$ is the Euler--Mascheroni constant. Meanwhile, extracting the $ O(\varepsilon^2)$ term from the associated Legendre differential equation \begin{align}\widehat{\mathscr L}_{\nu,\varepsilon}^{}
 P_\nu^{\varepsilon}(1-2t)\colonequals \frac{\partial }{\partial t}\left[ t(1-t)\frac{\partial P_\nu^{\varepsilon}(1-2t)}{\partial t} \right]+\left[\nu(1+\nu)-\frac{\varepsilon^{2}}{4t(1-t)}\right]P_\nu^{\varepsilon}(1-2t)=0,
\end{align}we obtain an inhomogeneous version of the Legendre  differential equation:\begin{align}
\widehat{\mathscr L}_{-1/2}\left.\!\frac{\partial^2 P_{-1/2}^{\varepsilon}(1-2t)}{\partial\varepsilon^2}\right|_{\varepsilon=0}=\frac{P_{-1/2}(1-2t)}{2t(1-t)}.
\end{align}Solving this inhomogeneous equation by variation of parameters, we get\begin{align}
\begin{split}\left.\!\frac{\partial^2 P_{-1/2}^{\varepsilon}(1-2t)}{\partial\varepsilon^2}\right|_{\varepsilon=0}={}&\frac{4}{\pi^{2}}\int_{1/2}^t\mathbf K\big(\sqrt{\mathstrut s}\big)\big[\mathbf K\big(\sqrt{\mathstrut 1-s}\big)\mathbf K\big(\sqrt{\mathstrut t}\big)-\mathbf K\big(\sqrt{\mathstrut s}\big)\mathbf K\big(\sqrt{\mathstrut 1-t}\big)\big]\frac{\D s}{s(1-s)}\\{}&+c_1\mathbf  K\big(\sqrt{\mathstrut t}\big)+c_2\mathbf  K\big(\sqrt{\mathstrut 1-t}\big),
\end{split}
\end{align}  where $c_1$ and $c_2$  are constants. One may deduce \begin{align}
c_{1}={}&-\frac{\pi}{2}+\frac{2(\gamma_0+2\log2)^2}{\pi}\intertext{and}c_2={}&\frac{8G}{\pi}-2(\gamma_0+2\log2)
\end{align} from (cf.\ \cite[\S3.4(20) and \S3.4(23), the latter being corrected in the errata]{HTF1})
\begin{align}
\left.\!\frac{\partial^2 P_{-1/2}^{\varepsilon}(0)}{\partial\varepsilon^2}\right|_{\varepsilon=0}={}&P_{-1/2}(0)\left[4G+(\gamma_{0} +2 \log 2)^2-\pi  (\gamma _{0}+2 \log 2)-\frac{\pi ^2}{4}\right]\intertext{and}
\left.\!\frac{\partial^3 P_{-1/2}^{\varepsilon}(x)}{\partial\varepsilon^2\partial x}\right|_{\raisebox{.3em}{\ensuremath{\substack{\varepsilon=0\\x=0}}}}
={}&\left.\!\frac{\partial P_{-1/2}^{}(x)}{\partial x}\right|_{x=0}\left[-4 G-\frac{\pi ^2}{4}+(\gamma_{0} +2 \log 2)^2+\pi  (\gamma_{0} +2 \log 2) -\frac{\pi ^2}{4}\right].
\end{align}
Thus far, we have verified  \eqref{eq:eps''} in its entirety.\end{proof}Later on, we will reparametrize \eqref{eq:nu''}  and \eqref{eq:eps''} with $ t=\alpha_4(z)=\lambda(2z)$, using the modular lambda function given in \eqref{eq:lambda_defn}, together with  a key relation \begin{align}
z=\frac{i\mathbf K\big(\sqrt{1-\lambda(z)}\big)}{\mathbf K\big(\sqrt{\lambda(z)}\big)},\quad\text{where } \I z>0,|\RE z|<1,\left|z+\frac{1}{2}\right|>\frac{1}{2},\left|z-\frac{1}{2}\right|>\frac{1}{2}.
\label{eq:zKratio}\end{align}
\subsection{Lambert--Ramanujan series, Eichler integrals, and Epstein zeta functions\label{subsec:modH2}}
The standard Lambert series take the form\begin{align}
\sum_{n=1}^\infty \frac{a_n x^n}{1-x^{n}}.
\end{align}The Lambert--Ramanujan series\footnote{These were simply called ``Ramanujan series'' in \cite{EZF}. Here, we use the hyphenated terminology ``Lambert--Ramanujan'' to avoid confusion with other series \cite{Cooper2017Theta,GuilleraRogers2014} bearing   Ramanujan's name.} refer to the special cases where $ a_n=\frac{1}{n^{r}}$ for $r\in\mathbb Z_{>0}$, as attested in Ramanujan's second notebook   \cite[pp.~276--277]{RN2}. The Lambert--Ramanujan series satisfy functional equations like (cf.\ \cite[p.~276]{RN2} or \cite[Lemma 2.2]{EZF})\begin{align}&
\begin{split}&\frac{1}{(z/i)^m}\sum_{n=1}^\infty\frac{1}{n^{2m+1}\big(e^{2n\pi\frac{z}{i} }-1\big)}-(-z/i)^{m}\sum_{n=1}^\infty\frac{1}{n^{2m+1}\big(e^{2n\pi\frac{i}{z} }-1\big)}\\={}&\frac{\zeta(2m+2)}{2\pi}\left[\frac{1}{(z/i)^{m+1}}+(-z/i)^{m+1}\right]-\frac{\zeta(2m+1)}{2}\left[\frac{1}{(z/i)^{m}}-(-z/i)^{m}\right]\\{}&+\frac{1}{\pi}\sum_{k=0}^{m-1}\frac{\zeta(2k+2)\zeta(2m-2k)}{(-1)^{k}(z/i)^{m-2k-1}},
\end{split}\label{eq:Ramanujan_reflection_notebook}\end{align}where $ z\in\mathfrak H,m\in\mathbb Z_{>0}$. For $m=1$, we may rewrite the functional equation  above as a sum rule for Eichler integrals [defined in \eqref{eq:EichlerE4defn}]\begin{align}
\mathscr E_4(z)-z^2\mathscr E_4\left( -\frac{1}{z} \right)=-\frac{z^4-5 z^2+1}{3 z}-\frac{30 \zeta (3) \left(z^2-1\right) }{\pi ^3i},\label{eq:EichlerE4refl}
\end{align}
because the Lambert series representation of $ E_4(z)$ [the last term in \eqref{eq:E4defn}] leaves us\begin{align}\mathscr
E_4(z)=\frac{60 i}{\pi^3}\sum_{n=1}^\infty\frac{1}{n^3(e^{-2\pi in z}-1)}.\label{eq:Eichler4_defn}\tag{\ref{eq:EichlerE4defn}$'$}
\end{align}By extension, we will also name \begin{align}
\sum_{n=0}^\infty \frac{1}{(2n+1)^{r}}\frac{ x^{2n+1}}{1-x^{2n+1}}=\sum_{n=1}^\infty \frac{1}{n^{r}}\frac{ x^n}{1-x^{n}}-\frac{1}{2^{r}}\sum_{n=1}^\infty \frac{1}{n^{r}}\frac{ x^{2n}}{1-x^{2n}}
\end{align}after Lambert and Ramanujan.\begin{proposition}
If  $z$ satisfies the following constraints:\begin{align}
\I z>0,\;|\RE z|\leq\frac12,\;|4\alpha_4(z)[1-\alpha_4(z)]|\leq1,\;|z|\geq\frac{1}{2},\label{eq:z_ineq_H2}
\end{align}then
 we have {\allowdisplaybreaks\begin{align}\frac{\displaystyle \sum_{k=0}^\infty\binom{2k}{k}^3\left[ \mathsf H_{2k}^{(2)}-\frac{1}{4}\mathsf H_{ k}^{(2)} \right]\left\{\frac{\alpha_{4}(z)[1-\alpha_{4}(z)]}{2^{4}}\right\}^k}{\displaystyle  \sum_{k=0}^\infty\binom{2k}{k}^3\left\{\frac{\alpha_{4}(z)[1-\alpha_{4}(z)]}{2^{4}}\right\}^k}={}&\sum_{n=0}^\infty\frac{2}{(2n+1)^{2}\cosh^2\frac{(2n+1)\pi z_{}}{i}},\label{eq:LR1}\end{align}and\begin{align}\begin{split}&\frac{\displaystyle \sum_{k=0}^\infty\binom{2k}{k}^3\mathsf H_{ k}^{(2)}\left\{\frac{\alpha_{4}(z)[1-\alpha_{4}(z)]}{2^{4}}\right\}^k}{\displaystyle  \sum_{k=0}^\infty\binom{2k}{k}^3\left\{\frac{\alpha_{4}(z)[1-\alpha_{4}(z)]}{2^{4}}\right\}^k}\\={}&\sum_{n=1}^\infty\frac{2}{n^{2}\sinh^2\frac{4n\pi z _{}}{i}}-\sum_{n=1}^\infty\frac{1}{2n^{2}\sinh^2\frac{2n\pi z _{}}{i}}+\sum_{n=0}^\infty\frac{2}{(2n+1)^{2}\sinh^2\frac{(2n+1)\pi z _{}}{i}},
\end{split}\label{eq:LR2}
\end{align}}where $\alpha_4(z) =\lambda(2z)$ [see \eqref{eq:lambda_defn}]. \end{proposition}\begin{proof}From \cite[(4.8)]{EZF}, we know that \begin{align}
\int_0^{\lambda(2z)}\frac{\big[\mathbf K\big(\sqrt{\smash[b]s}\big)\big]^2}{8}\left[ \frac{i\mathbf K\big(\sqrt{1-s}\big)}{\mathbf K\big(\sqrt{s}\big)} -2z\right]^2\D s=-\sum_{n=0}^\infty\frac{1}{(2n+1)^{3}\big[e^{(2n+1)\pi \frac{2z}{i}}+1\big]}\frac{}{}
\end{align}for $z/i>0$. Differentiating both sides of the equation above in $z$, before recalling  \cite[(3.15)]{Zhou2025NotesBHS}\footnote{In view of \eqref{eq:zKratio}, \eqref{eq:z_ineq_H2}, and \eqref{eq:KK}, the denominators on the left-hand sides of \eqref{eq:LR1} and \eqref{eq:LR2} never vanish, thus
making the two quotients well-defined. }\begin{align}
\sum_{k=0}^\infty\binom{2k}{k}^3\left[\frac{t(1-t)}{2^{4}}\right]^k=\left( \frac{2}{\pi} \right)^{2}\big[\mathbf K\big(\sqrt{t}\big)\big]^2,\quad \text{where }|4t(1-t)|<1,\RE t<\frac{1}{2}\label{eq:KK}
\end{align} along with  \eqref{eq:nu''} and \eqref{eq:zKratio}, we are able to confirm \eqref{eq:LR1} for  $2z/i\geq1$ \big[such that $ 0<\lambda(2z)\leq\frac12$\big]. For $z$ fulfilling the requirements in  \eqref{eq:z_ineq_H2}, one can prove by analytic continuation.\footnote{Note that the left-hand side of  \eqref{eq:LR1} cannot be analytically continued across the barrier $ |z|=\frac12$, on which $ |4\alpha_4(z)[1-\alpha_4(z)]|\geq1$. Therefore, although the left-hand side of  \eqref{eq:LR1} is ostensibly invariant under the transformation $ z\mapsto -\frac{1}{4z}$, the right-hand side of \eqref{eq:LR1}  does not enjoy such an invariance. The same can be said for \eqref{eq:LR2}.}

As a precursor to the verification of \eqref{eq:LR2}, we reformulate the integral representation in   \eqref{eq:eps''} as{\allowdisplaybreaks\begin{align}
\begin{split}&\sum_{k=0}^\infty\binom{2k}{k}^3\mathsf H_{ k}^{(2)}\left[\frac{t(1-t)}{2^{4}}\right]^k\\={}&\left(\frac{2}{\pi}\right)^{3}\int_{0}^t\mathbf K\big(\sqrt{\mathstrut s}\big)\mathbf K\big(\sqrt{\mathstrut t}\big)\big[\mathbf K\big(\sqrt{\mathstrut 1-s}\big)\mathbf K\big(\sqrt{\mathstrut t}\big)-\mathbf K\big(\sqrt{\mathstrut s}\big)\mathbf K\big(\sqrt{\mathstrut 1-t}\big)\big]\frac{\D s}{1-s}\\&-\left(\frac{2}{\pi}\right)^{3}\int_{t}^1\mathbf K\big(\sqrt{\mathstrut s}\big)\mathbf K\big(\sqrt{\mathstrut t}\big)\big[\mathbf K\big(\sqrt{\mathstrut 1-s}\big)\mathbf K\big(\sqrt{\mathstrut t}\big)-\mathbf K\big(\sqrt{\mathstrut s}\big)\mathbf K\big(\sqrt{\mathstrut 1-t}\big)\big]\frac{\D s}{s}\\{}&-\frac{\big[\mathbf K\big(\sqrt{\mathstrut t}\big)\big]^{2}}{3}-\big[\mathbf K\big(\sqrt{\mathstrut 1-t}\big)\big]^{2}+\frac{28\mathbf K\big(\sqrt{\mathstrut t}\big)\mathbf K\big(\sqrt{\mathstrut 1-t}\big)\zeta(3)}{\pi^{3}}.\end{split}\label{eq:'eps''}\tag{\ref{eq:eps''}$'$}
\end{align}}To justify this reformulation, it suffices to show that\begin{align}
\int_{0}^{1/2}\frac{[\mathbf K(\sqrt{\smash[b]s})]^2}{1-s}\D s-\int_{1/2}^{1}\frac{[\mathbf K(\sqrt{\smash[b]s})]^2}{s}\D s=\frac{7\zeta(3)}{2}-2\pi G.
\end{align}      Towards this end, we note that \cite[(4.7)]{EZF}\begin{align}
\int_0^{\lambda(2z)}\frac{\big[\mathbf K\big(\sqrt{\smash[b]s}\big)\big]^2}{8}\left[ \frac{i\mathbf K\big(\sqrt{1-s}\big)}{\mathbf K\big(\sqrt{s}\big)} -2z\right]^2\frac{\D s}{1-s}=-\sum_{n=0}^\infty\frac{1}{(2n+1)^{3}\big[e^{(2n+1)\pi \frac{2z}{i}}-1\big]}\label{eq:sinh_prep}
\end{align}holds for  $z/i>0$, and its leading-order Taylor expansion around $ z=\frac{i}{2}$ yields a pair of identities:{\allowdisplaybreaks\begin{align}
\begin{split}&8\sum_{n=0}^\infty\frac{1}{(2n+1)^{3}\big[e^{(2n+1)\pi }-1\big]}=\int_0^{1/2}\frac{\big[\mathbf K\big(\sqrt{\smash[b]s}\big)\big]^2}{1-s}\left[ \frac{\mathbf K\big(\sqrt{1-s}\big)}{\mathbf K\big(\sqrt{s}\big)} -1\right]^2\D s\\={}&\int_{0}^{1/2}\frac{\big[\mathbf K\big(\sqrt{\smash[b]s}\big)\big]^2}{1-s}\D s+\int_{1/2}^{1}\frac{\big[\mathbf K\big(\sqrt{\smash[b]s}\big)\big]^2}{s}\D s-2\int_{0}^{1/2}\frac{\mathbf K\big(\sqrt{\smash[b]s}\big)\mathbf K\big(\sqrt{1-s}\big)}{1-s}\D s,
\end{split}\\\begin{split}&\pi\sum_{n=0}^\infty\frac{1}{(2n+1)^{2}\sinh^2\frac{(2n+1)\pi _{}}{2}}=\int_0^{1/2}\frac{\big[\mathbf K\big(\sqrt{\smash[b]s}\big)\big]^2}{1-s}\left[ \frac{\mathbf K\big(\sqrt{1-s}\big)}{\mathbf K\big(\sqrt{s}\big)} -1\right]\D s\\={}&\int_{0}^{1/2}\frac{\mathbf K\big(\sqrt{\smash[b]s}\big)\mathbf K\big(\sqrt{1-s}\big)}{1-s}\D s-\int_{0}^{1/2}\frac{\big[\mathbf K\big(\sqrt{\smash[b]s}\big)\big]^2}{1-s}\D s,
\end{split}
\end{align}}which combine into\begin{align}
\begin{split}&\int_{1/2}^{1}\frac{\big[\mathbf K\big(\sqrt{\smash[b]s}\big)\big]^2}{s}\D s-\int_{0}^{1/2}\frac{\big[\mathbf K\big(\sqrt{\smash[b]s}\big)\big]^2}{1-s}\D s\\={}&8\sum_{n=0}^\infty\frac{1}{(2n+1)^{3}\big[e^{(2n+1)\pi }-1\big]}+2\pi\sum_{n=0}^\infty\frac{1}{(2n+1)^{2}\sinh^2\frac{(2n+1)\pi _{}}{2}}\\={}&-\frac{7\zeta(3)}{2}+\frac{2\pi^{3}}{3}\frac{4E\big(\frac{i}{2},2\big)-E(i,2)}{60}.
\end{split}
\end{align}   Here in the last step, we have set $z=\frac{i}{2} $ in the following formula applicable to all $z\in\mathfrak H$:\begin{align}
\begin{split}&\frac{4E(z,2)-E(2z,2)}{60}\\={}&\frac{21 \zeta (3)}{8 \pi ^3\I z}+\frac{6}{\pi^3\I z}\RE\sum_{n=0}^\infty\frac{1}{(2n+1)^{3}\big[e^{(2n+1)\pi\frac{2z}{i} }-1\big]}+\frac{3}{\pi^{2}}\RE\sum_{n=0}^\infty\frac{1}{(2n+1)^{2}\sinh^2\frac{(2n+1)\pi z _{}}{i}},\end{split}\label{eq:EZF_LR}
\end{align}which in turn, descends from special cases of \cite[(1.9) and (3.5)]{EZF}. Since the evaluations [cf.\ \eqref{eq:EZF_defn}]\begin{align}
E\left(\frac{i}{2},2\right)=E(2i,2)=\frac{105G}{2\pi^{2}}\text{ and }E(i,2)=\frac{30G}{\pi^{2}}
\end{align}can be found in   \cite[Table VI]{GlasserZucker1980} (also reproduced in Table \ref{tab:H2} of this article), we have completed our justification of \eqref{eq:'eps''}. Now, appealing again to the leading-order Taylor expansion of \eqref{eq:sinh_prep}, we have\begin{align}
\int_0^{\lambda(2z)}\frac{\big[\mathbf K\big(\sqrt{\smash[b]s}\big)\big]^2}{1-s}\left[ \frac{i\mathbf K\big(\sqrt{1-s}\big)}{\mathbf K\big(\sqrt{s}\big)} -2z\right]\D s={}&\pi i\sum_{n=0}^\infty\frac{1}{(2n+1)^{2}\sinh^2\frac{(2n+1)\pi z _{}}{i}}
\end{align}and{\allowdisplaybreaks\begin{align}
\begin{split}&\int^1_{\lambda(2z)}\frac{\big[\mathbf K\big(\sqrt{\smash[b]s}\big)\big]^2}{s}\left[ \frac{i\mathbf K\big(\sqrt{1-s}\big)}{\mathbf K\big(\sqrt{s}\big)} -2z\right]\D s=\int_0^{\lambda\left( -\frac{1}{2z} \right)}\frac{\big[\mathbf K\big(\sqrt{\smash[b]1-s}\big)]^2}{1-s}\left[ \frac{i\mathbf K\big(\sqrt{s}\big)}{\mathbf K\big(\sqrt{1-s}\big)} -2z\right]\D s\\={}&2z\int_0^{\lambda\left( -\frac{1}{2z} \right)}\frac{\big[\mathbf K\big(\sqrt{\smash[b]s}\big)\big]^2}{1-s}\left[ \frac{i\mathbf K\big(\sqrt{1-s}\big)}{\mathbf K\big(\sqrt{s}\big)} +\frac{1}{2z}\right]^{2}\D s\\{}&-\int_0^{\lambda\left( -\frac{1}{2z} \right)}\frac{\big[\mathbf K\big(\sqrt{\smash[b]s}\big)\big]^2}{1-s}\left[ \frac{i\mathbf K\big(\sqrt{1-s}\big)}{\mathbf K\big(\sqrt{s}\big)} +\frac{1}{2z}\right]\D s\\={}&-16z\sum_{n=0}^\infty\frac{1}{(2n+1)^{3}\big[e^{(2n+1)\frac{\pi i}{2z} }-1\big]}-\pi i\sum_{n=0}^\infty\frac{1}{(2n+1)^{2}\sinh^2\frac{(2n+1)\pi i _{}}{4z}}\frac{}{}\end{split}
\end{align}}for $z/i>0$. Thus, we have identified the left-hand side of  \eqref{eq:LR2} with\begin{align}
\begin{split}&\frac{14 z \zeta (3)}{\pi i}+\left( z^2 -\frac{1}{12}\right)\pi^2+\frac{z}{\pi}\sum_{n=0}^\infty\frac{32}{(2n+1)^{3}\big[e^{(2n+1)\frac{\pi i}{2z} }-1\big]}\\&+\sum_{n=0}^\infty\frac{2}{(2n+1)^{2}\sinh^2\frac{(2n+1)\pi i _{}}{4z}}+\sum_{n=0}^\infty\frac{2}{(2n+1)^{2}\sinh^2\frac{(2n+1)\pi z _{}}{i}}.\end{split}
\end{align}This expression is clearly real-valued for    $2z/i>1 $ [as is  the left-hand side of  \eqref{eq:LR2}], but not manifestly so when $   \RE z=\frac12$ and $ \I z\geq\frac{1}{\sqrt{2}}$.  As an amendment for the latter, we differentiate the  reflection formula \eqref{eq:EichlerE4refl}  with respect to $z$, so as to deduce\begin{align}
\begin{split}&\frac{14 z \zeta (3)}{\pi i}+\left( z^2 -\frac{1}{12}\right)\pi^2+\frac{z}{\pi}\sum_{n=0}^\infty\frac{32}{(2n+1)^{3}\big[e^{(2n+1)\frac{\pi i}{2z} }-1\big]}+\sum_{n=0}^\infty\frac{2}{(2n+1)^{2}\sinh^2\frac{(2n+1)\pi i _{}}{4z}}\\={}&\sum_{n=1}^\infty\frac{2}{n^{2}\sinh^2\frac{4n\pi z _{}}{i}}-\sum_{n=1}^\infty\frac{1}{2n^{2}\sinh^2\frac{2n\pi z _{}}{i}}.
\end{split}
\end{align}This concludes our proof of \eqref{eq:LR2}  for   $2z/i>1 $, as well as (after analytic continuation) for $z$ in  the general setting  \eqref{eq:z_ineq_H2}.
\end{proof}

With the foregoing preparations, we are ready for the verification of our first major result  in this article.\begin{proof}[Proof of Theorem \ref{thm:H2}]One may translate the inequality $ 2z/i\geq1$ \big(resp.\ the relations $ \RE z=\frac12$ and $ \I z\geq\frac{1}{\sqrt{2}}$\big) into $  4\alpha_4(z)[1-\alpha_4(z)]\in(0,1]$ [resp.\ $4\alpha_4(z)[1-\alpha_4(z)]\in[-1,0)$], so the conditions in  \eqref{eq:z_ineq_H2} are met.

\begin{enumerate}[leftmargin=*,  label=(\alph*),ref=(\alph*),
widest=d, align=left] \item
Comparing  \eqref{eq:LR1} and \eqref{eq:EZF_LR},  we get {\allowdisplaybreaks\begin{align}\begin{split}\mathscr Q_1(z)\colonequals{} &
\frac{\displaystyle \sum_{k=0}^\infty\binom{2k}{k}^3\left[ \mathsf H_{2k}^{(2)}-\frac{1}{4}\mathsf H_{ k}^{(2)} \right]\left\{\frac{\alpha_{4}(z)[1-\alpha_{4}(z)]}{2^{4}}\right\}^k}{\displaystyle  \sum_{k=0}^\infty\binom{2k}{k}^3\left\{\frac{\alpha_{4}(z)[1-\alpha_{4}(z)]}{2^{4}}\right\}^k}\\={}&\frac{7\zeta(3)}{4\pi\I z }-\frac{\pi^2\left[4 E \big(z+\frac12,2\big)-E(2z,2)\right]}{90}\\{}&+\frac{4}{\pi\I z}\sum_{n=0}^\infty\frac{}{}\frac{1}{(2n+1)^{3}\big[e^{(2n+1)\frac{\pi (2z+1)}{i} }-1\big]}.
\end{split}
\end{align}}Referring back to \eqref{eq:Eichler4_defn}, we may convert the last infinite series into a linear combination of two Eichler integrals, as declared in \eqref{eq:Q1}.

Putting \eqref{eq:LR2} together with \eqref{eq:EZF_LR} and its analog  \cite[(1.8), $ m=1$]{EZF}\begin{align}
\begin{split}E(z,2)={}&(\I z)^2+\frac{45\zeta(3)}{\pi^3\I z}+\frac{90}{\pi^3\I z}\RE\sum_{n=1}^\infty\frac{1}{n^{3}\big(e^{\frac{2n\pi z}{i}}-1\big)}\\&+\frac{45}{\pi^{2}}\RE\sum_{n=1}^\infty\frac{1}{n^{2}\sinh^2\frac{n\pi z _{}}{i}},\quad z\in\mathfrak H,
\end{split}\label{eq:E(z,2)LR}
\end{align} we obtain{\allowdisplaybreaks
\begin{align}
\begin{split}\mathscr Q_2(z)\colonequals{} &\frac{\displaystyle \sum_{k=0}^\infty\binom{2k}{k}^3\mathsf H_{ k}^{(2)}\left\{\frac{\alpha_{4}(z)[1-\alpha_{4}(z)]}{2^{4}}\right\}^k}{\displaystyle  \sum_{k=0}^\infty\binom{2k}{k}^3\left\{\frac{\alpha_{4}(z)[1-\alpha_{4}(z)]}{2^{4}}\right\}^k}\\={}&-\frac{2\pi^2}{3}(\I z)^2-\frac{2\zeta(3)}{\pi\I z }+\frac{\pi^2\left[2 E (z,2)-E(2z,2)+2E(4z,2)\right]}{45}\\&{}+\frac{1}{2\pi \I z}\sum_{n=1}^\infty\frac{1}{n^{3}\big(e^{\frac{4n\pi z}{i}}-1\big)}-\frac{1}{\pi \I z}\sum_{n=1}^\infty\frac{1}{n^{3}\big(e^{\frac{8n\pi z}{i}}-1\big)}\\&{}-\frac{4}{\pi\I z}\sum_{n=0}^\infty\frac{}{}\frac{1}{(2n+1)^{3}\big[e^{(2n+1)\frac{\pi (2z)}{i} }-1\big]},\end{split}
\end{align}}the right-hand side of which is equivalent to\begin{align}
\begin{split}&-\frac{2\pi^2}{3}(\I z)^2-\frac{2\zeta(3)}{\pi\I z }+\frac{\pi^2\left[2 E (z,2)-E(2z,2)+2E(4z,2)\right]}{45}\\&{}+\frac{\pi^2 i\left[4\mathscr  E_{4} (z)-\mathscr  E_{4}(2z)+\mathscr  E_{4}(4z)\right]}{60\I z},
\end{split}
\end{align} by virtue of  \eqref{eq:Eichler4_defn}. To bridge the gap between the last displayed formula and the statement in \eqref{eq:Q2}, we need to manipulate series of the  Lambert(--Ramanujan) type, to establish\begin{align}
0={}&E_4\left( z+\frac{1}{2} \right)+E_4(z)-18E_{4}(2z)+16E_4(4z)\label{eq:sumE4}\intertext{and}0={}&4\mathscr E_4\left( z+\frac{1}{2} \right)+4\mathscr E_4(z)-9\mathscr E_{4}(2z)+\mathscr E_4(4z)\label{eq:sumEichlerE4}
\end{align}for all $ z\in\mathfrak H$, before putting \eqref{eq:sumE4} into the context of \begin{align}
E(z,2)={}&(\I z)^2+\frac{45\zeta(3)}{\pi^3\I z}-\frac{3}{2\I z}\RE\left\{ i\int_z^{i\infty}[1-E_4(w)](w-z)(w-\overline z)\D w \right\}
\tag{\ref{eq:E(z,2)LR}$'$}\label{eq:E(z,2)LR'}
\end{align} and constructing a sum rule for the Epstein zeta functions\begin{align}
0=2E\left( z+\frac{1}{2} ,2\right)+2E(z,2)-9E(2z,2)+2E(4z,2)\label{eq:E(z,2)add}
\end{align}  that is applicable to all $ z\in\mathfrak H$.
\item By \cite[(2.17) and (3.25)]{Zhou2025NotesBHS}, we have \begin{align}\begin{split}&
\frac{\partial}{\partial \I z }\frac{\left\{\alpha_{4}(z)[1-\alpha_{4}(z)]\right\}^{k}}{\big[\mathbf K\big(\sqrt{\alpha_4(z)}\big)\big]^2\I z}\\={}&-\frac{4}{\pi}\frac{2[1-2\alpha_4(z)]k+ R_{-1/2}(1-2\alpha_4(z))}{\I z}\left\{\alpha_{4}(z)[1-\alpha_{4}(z)]\right\}^{k}.
\end{split}
\end{align}Simplifying the denominator of $ \mathscr Q_1(z)$ with \eqref{eq:KK}, and referring to the last displayed equation, we arrive at \begin{align}
\begin{split}&\mathscr R_1(z)\\\colonequals {}&\sum_{k=0}^\infty\binom{2k}{k}^3\left[ \mathsf H_{2k}^{(2)}-\frac{1}{4}\mathsf H_{ k}^{(2)} \right]\frac{2[1-2\alpha_4(z)]k+ R_{-1/2}(1-2\alpha_4(z))}{\I z}\left\{\frac{\alpha_{4}(z)[1-\alpha_{4}(z)]}{2^{4}}\right\}^k
\\={}&\frac{\mathscr Q_1(z)}{\pi(\I z)^2}+\frac{4}{\I z}\sum_{n=0}^\infty\frac{\tanh\frac{(2n+1)\pi z _{}}{i}}{(2n+1)\cosh^2\frac{(2n+1)\pi z}{i}}.
\end{split}\label{eq:R1(z)''}
\end{align}To convert the remaining series into the $ \mathscr E''_4$ expressions in  \eqref{eq:R1(z)},  twice differentiate \eqref{eq:Eichler4_defn}.

Arguing in a similar manner, we  have{\allowdisplaybreaks
\begin{align}
\begin{split}&\mathscr R_2(z)\\\colonequals {}&\sum_{k=0}^\infty\binom{2k}{k}^3\mathsf H_{ k}^{(2)}\frac{2[1-2\alpha_4(z)]k+ R_{-1/2}(1-2\alpha_4(z))}{\I z}\left\{\frac{\alpha_{4}(z)[1-\alpha_{4}(z)]}{2^{4}}\right\}^k
\\={}&\frac{\mathscr Q_2(z)}{\pi(\I z)^2}-\frac{2}{\I z}\sum_{n=1}^\infty\frac{\coth\frac{2n\pi z _{}}{i}}{n\sinh^2\frac{2n\pi z}{i}}+\frac{16}{\I z}\sum_{n=1}^\infty\frac{\coth\frac{4n\pi z _{}}{i}}{n\sinh^2\frac{4n\pi z}{i}}\\{}&+\frac{4}{\I z}\sum_{n=0}^\infty\frac{\coth\frac{(2n+1)\pi z _{}}{i}}{(2n+1)\sinh^2\frac{(2n+1)\pi z}{i}},\end{split}
\end{align}
}which implies   \eqref{eq:R2(z)}, thanks to   the second-order derivatives of both \eqref{eq:Eichler4_defn} and \eqref{eq:sumEichlerE4}.
    \qedhere\end{enumerate} \end{proof}\subsection{Special  sum rules for $ \mathscr E_4$ and derivatives\label{subsec:sumEichler}}Now, we evaluate $ \mathscr S_{r(z)}(z)$ [defined in \eqref{eq:Sr(z)}] and  $ \mathscr T_{r(z)}(z)$ [defined in \eqref{eq:Tr(z)}] for some special points  $z$. In this process, we will repeatedly invoke the sum rules  \eqref{eq:EichlerE4refl} and  \eqref{eq:sumEichlerE4} for the Eichler integrals, as well as (tacitly) the translational invariance $ \mathscr E_4(z)=\mathscr E_4(z+1)$. These efforts will amount to the final proofs of \eqref{eq:Sun1}--\eqref{eq:Sun4} in Proposition \ref{prop:E4sum} and refreshed views of \eqref{eq:-64H2}--\eqref{eq:4096H2} in Proposition \ref{prop:E4''sum}.\begin{proposition}\label{prop:E4sum}
\begin{enumerate}[leftmargin=*,  label=\emph{(\alph*)},ref=(\alph*),
widest=d, align=left] \item
We have \begin{align}
\mathscr E_4\left( \frac{1+\sqrt{3}i}{2} \right)=\frac{2i}{\sqrt{3}}+\frac{30\zeta(3)}{\pi^3i},\label{eq:EichlerE4isqrt3}
\end{align} which entails\begin{align}
\mathscr S_{1/16}\left( \frac{\sqrt{3}i}{2} \right)=\frac{11\pi^{2}}{96}.\label{eq:S16isqrt3}
\end{align}
\item We have \begin{align}
12\mathscr E_4\left( \frac{1+\sqrt{7} i}{2}\right)-\mathscr E_4\left( \sqrt{7} i\right)=\frac{29 \sqrt{7} i}{6}+\frac{330  \zeta (3)}{\pi ^3i},\label{eq:EichlerE4isqrt7}
\end{align}which entails \begin{align}
\mathscr S_{1/46}\left( \frac{\sqrt{7}i}{2} \right)=\frac{43\pi^{2}}{552}.\label{eq:S46isqrt7}
\end{align}
\item We have \begin{align}2\mathscr E_4\left( \frac{i}{\sqrt{2}} \right)+
\mathscr E_4\left(\sqrt{2}i\right)=\frac{5i}{\sqrt{2}}+\frac{90\zeta(3)}{\pi^3i},\label{eq:EichlerE4isqrt2}
\end{align} which entails\begin{align}
\mathscr S_{1/4}\left( \frac{1}{2} +\frac{i}{\sqrt{2}}\right)=\frac{5\pi^{2}}{24}.\label{eq:S4isqrt2}
\end{align}
\item We have \begin{align}
\mathscr E_4(i)=\frac{7i}{6}+\frac{30\zeta(3)}{\pi^3i},\label{eq:EichlerE4i}
\end{align} which entails\begin{align}
\mathscr S_{1/16}\left( \frac{1}{2} +i\right)=\frac{11\pi^{2}}{96}.\label{eq:S16i}
\end{align}
\end{enumerate}\end{proposition}\begin{proof}\begin{enumerate}[leftmargin=*,  label=(\alph*),ref=(\alph*),
widest=d, align=left] \item
To verify \eqref{eq:EichlerE4isqrt3}, plug $ z=\frac{1+\sqrt{3}i}{2}$ into \eqref{eq:EichlerE4refl}.

For any $ z\in\mathfrak H$, we have \begin{align}
\mathscr S_{1/16}(z)=\frac{\pi^{2}}{24}(\I z)^{2}+\frac{15\zeta(3)}{8\pi\I z}-\frac{\pi^2i\mathscr  E_{4} \big(z+\frac12\big)}{16}.\label{eq:S16}
\end{align}Specializing the equation above  to $ z=\frac{\sqrt{3}i}{2}$, we get \eqref{eq:S16isqrt3}.
\item For  $ z=\frac{1+\sqrt{7}i}2$, we have\begin{align}
\frac{1}{2}-\frac{1}{z}=\frac{1}{1-z},\quad -\frac{2}{z}=z-1,\quad -\frac{4}{z}=2z-2,
\end{align}   so \eqref{eq:sumEichlerE4} brings us \begin{align}
0={}&4\mathscr E_4\left( \frac{1}{1-z} \right)+4\mathscr E_4\left( -\frac{1}{z} \right)-9\mathscr E_{4}(z)+\mathscr E_4(2z).
\end{align}With the aid of  \eqref{eq:EichlerE4refl}, we may invert the arguments of the first two  Eichler integrals on the right-hand side of the equation above, thereby confirming  \eqref{eq:EichlerE4isqrt7}.
Here, the coefficient $12$ comes from the algebra
\begin{align}
\frac{4}{(1-z)^{2}}+\frac{4}{z^2}-9=-12
\end{align}for  $ z=\frac{1+\sqrt{7}i}2$.

We have\begin{align}
\mathscr S_{1/46}(z)=\frac{\pi^{2}}{69}(\I z)^{2}+\frac{165\zeta(3)}{92\pi\I z}-\frac{\pi^2i\left[12\mathscr  E_{4} \big(z+\frac12\big)-\mathscr E_4(2z)\right]}{184}
\end{align} for any  $ z\in\mathfrak H$. Specializing the equation above  to $ z=\frac{\sqrt{7}i}{2}$, we get \eqref{eq:S46isqrt7}.\item To verify \eqref{eq:EichlerE4isqrt2}, plug $ z=\frac{i}{\sqrt{2}}$ into \eqref{eq:EichlerE4refl}.

We have\begin{align}
\mathscr S_{1/4}(z)=\frac{\pi^{2}}{6}(\I z)^{2}+\frac{9\zeta(3)}{4\pi\I z}-\frac{\pi^2i\left[2\mathscr  E_{4} \big(z+\frac12\big)+\mathscr E_4(2z)\right]}{40}
\end{align} for any  $ z\in\mathfrak H$. Specializing the last displayed equation   to $ z=\frac{1}{2} +\frac{i}{\sqrt{2}}$, we get \eqref{eq:S4isqrt2}.
\item To verify \eqref{eq:EichlerE4i}, plug $ z=i$ into \eqref{eq:EichlerE4refl}.

Specializing \eqref{eq:S16} to $ z=\frac{1}{2}+i$, we get \eqref{eq:S16i}.
\qedhere\end{enumerate}\end{proof}
\begin{proposition}\label{prop:E4''sum}\begin{enumerate}[leftmargin=*,  label=\emph{(\alph*)},ref=(\alph*),
widest=d, align=left] \item We have \begin{align}
\mathscr E_4''\left( \frac{1+\sqrt{3}i}{2} \right)=-\frac{15\sqrt{3}L_{-3}(2)}{\pi^{2}i}-\sqrt{3}i,\label{eq:EichlerE4''isqrt3}
\end{align} which entails\begin{align}
\mathscr T_{1/16}\left( \frac{\sqrt{3}i}{2} \right)=\frac{\pi}{36}.\label{eq:T16isqrt3}
\end{align}
\item We have \begin{align}
3\mathscr E_4''\left( \frac{1+\sqrt{7}i}{2} \right)-\mathscr E_4''\left( \sqrt{7} i\right)=-\frac{35\sqrt{7}L_{-7}(2)}{4\pi^2i}-\sqrt{7}i,\label{eq:EichlerE4''isqrt7}
\end{align}which entails\begin{align}
\mathscr T_{1/46}\left( \frac{\sqrt{7}i}{2} \right)=\frac{\pi}{966}.\label{eq:T46isqrt7}
\end{align}
\item We have \begin{align}
\mathscr E_4''\left( \frac{i}{\sqrt{2}} \right)+2\mathscr E_4''\left( \sqrt{2} i\right)=-\frac{40\sqrt{2}L_{-8}(2)}{\pi^2i}-5\sqrt{2}i,\label{eq:EichlerE4''isqrt2}
\end{align}which entails\begin{align}
\mathscr T_{1/4}\left( \frac{1}{2} +\frac{i}{\sqrt{2}}\right)=-\frac{\pi}{12}.\label{eq:T4isqrt2}
\end{align}
\item We have \begin{align}
\mathscr E_4''(i)=-\frac{20L_{-4}(2)}{\pi^2i}-2i,\label{eq:EichlerE4''i}
\end{align}which entails\begin{align}
\mathscr T_{1/16}\left( \frac{1}{2} +i\right)=-\frac{\pi}{96}.\label{eq:T16i}
\end{align}
\end{enumerate}
\end{proposition}\begin{proof}\begin{enumerate}[leftmargin=*,  label=(\alph*),ref=(\alph*),
widest=d, align=left] \item The second-order derivative of  \eqref{eq:EichlerE4refl} reads\begin{align}
\mathscr E_4'' (z)-\frac{1}{z^{2}}\mathscr E_4'' \left( -\frac{1}{z} \right)-2\mathscr E_4 ^{}\left(- \frac{1}{z} \right)-\frac{2}{z}\mathscr E_4 '\left(- \frac{1}{z} \right)=-\frac{2}{3z^{3}}-2z-\frac{60\zeta(3)}{\pi^{3}i}\label{eq:EichlerE4''refl}
\end{align} for any $ z\in\mathfrak H$, where the integral representation of  $\mathscr E_4' $ (resp.\ $ \mathscr E''_4$) is given by\begin{align}
\mathscr E'_4(z)\colonequals \frac{\partial\mathscr
E_4(z)}{\partial z}=2\int_{z}^{i\infty} [1-E_4(w)](z-w)\D w\label{eq:EichlerE4'}
\end{align}[resp.\
 \eqref{eq:EichlerE4''}]. Setting $z=\frac{-1+\sqrt{3}i}{2} $ in the functional equation \eqref{eq:EichlerE4''refl}, we get \begin{align}\begin{split}&
\mathscr E_4''\left( \frac{-1+\sqrt{3}i}{2} \right)+\frac{1-\sqrt{3}i}{2}\mathscr E_4''\left( \frac{1+\sqrt{3}i}{2} \right)-2\mathscr E_4 ^{}\left(\frac{1+\sqrt{3}i}{2} \right)+\big(1+\sqrt{3}i\big)\mathscr E_4 '\left(\frac{1+\sqrt{3}i}{2} \right)\\={}&\frac{1}{3}-\sqrt{3}i-\frac{60\zeta(3)}{\pi^3i},
\end{split}\end{align}whose imaginary part leaves us [cf.\ \eqref{eq:E(z,2)LR'}]\begin{align}
\begin{split}\frac{3}{2}\mathscr E_4''\left( \frac{1+\sqrt{3}i}{2} \right)+\frac{\sqrt{3}i}{2}-\frac{60\zeta(3)}{\pi^3i}-\frac{2i}{\sqrt{3}}E\left( \frac{1+\sqrt{3}i}{2},2 \right)=-\sqrt{3}i-\frac{60\zeta(3)}{\pi^3i}.
\end{split}
\end{align}This is equivalent to \eqref{eq:EichlerE4''isqrt3}, thanks to the closed-form evaluation $E\left( \frac{1+\sqrt{3}i}{2},2 \right)=\frac{135L_{-3}(2)}{4\pi^{2}} $ (cf.\ Table \ref{tab:H2}).

Specializing\begin{align}
\mathscr T_{1/16}\left( z \right)=\frac{\mathscr Q_1(z)-\frac{1}{16}\mathscr Q_2(z)}{\pi(\I z)^2}-\frac{\pi i}{16\I z}\mathscr E''_4\left( z+\frac{1}{2} \right)\label{eq:T16}
\end{align}to $ z=\frac{\sqrt{3}i}{2}$, while recalling the proven evaluation $ \mathscr  Q_1\left( \frac{1+\sqrt{3}i}{2} \right)-\frac{1}{16} \mathscr  Q_2\left( \frac{1+\sqrt{3}i}{2} \right)=\frac{11\pi^{2}}{96}-\frac{45L_{-3}(2)}{32}$ from the penultimate column of Table \ref{tab:H2}, we reach \eqref{eq:T16isqrt3}.\item Twice differentiating   \eqref{eq:sumEichlerE4} in $z$, we may establish \begin{align}0=\mathscr E_4''\left( z+\frac{1}{2} \right)+\mathscr E_4''(z)-9\mathscr E_{4}''(2z)+4\mathscr E_4''(4z)\end{align}for any $ z\in\mathfrak H$. When  $ z=-\frac{2}{1+\sqrt{7}i}$, this boils down to \begin{align}
\begin{split}0={}&\mathscr E_4''\left( \frac{1+\sqrt{7 }i}{4} \right)+\mathscr E_4''\left( \frac{-1+\sqrt{7 }i}{4} \right)-9\mathscr E_{4}''\left( \frac{-1+\sqrt{7}i}{2} \right)+4\mathscr E_4''\left( -1+\sqrt{7} i\right)\\={}&\mathscr E_4''\left( -\frac{2}{-1+\sqrt{7}i} \right)+\mathscr E_4''\left( -\frac{2}{1+\sqrt{7}i} \right)-9\mathscr E_{4}''\left( \frac{1+\sqrt{7}i}{2} \right)+4\mathscr E_4''\left( \sqrt{7} i\right).
\end{split}
\end{align}Applying \eqref{eq:EichlerE4''refl}   to   $ z=-\frac{2}{\pm 1+i\sqrt{7}}$, we receive\begin{align}
\begin{split}&\mathscr E_4''\left( -\frac{2}{\pm1+\sqrt{7}i} \right)+\frac{3\mp\sqrt{7}i}{2}\mathscr E_4''\left( \frac{1+\sqrt{7}i}{2} \right)\\{}&-2\mathscr E_4^{}\left( \frac{1+\sqrt{7}i}{2} \right)+\big(\!\pm1+\sqrt{7}i\big)\mathscr E_4'\left( \frac{1+\sqrt{7}i}{2} \right)\\={}&\mp\frac{7}{6}-\frac{5\sqrt{7}i}{6}-\frac{60\zeta(3)}{\pi^{3}i}
\end{split}
\end{align} in return. Extracting the  imaginary part, we get \begin{align}\begin{split}&
\mathscr E_4''\left( -\frac{2}{-1+\sqrt{7}i} \right)+\mathscr E_4''\left( -\frac{2}{1+\sqrt{7}i} \right)+3\mathscr E_4''\left( \frac{1+\sqrt{7}i}{2} \right)\\&{}-4\mathscr E_4^{}\left( \frac{1+\sqrt{7}i}{2} \right)+2\sqrt{7}i\mathscr E_4'\left( \frac{1+\sqrt{7}i}{2} \right)\\={}&-\frac{5\sqrt{7}i}{3}-\frac{120\zeta(3)}{\pi^{3}i}.
\end{split}
\end{align} Now that \begin{align}
-4\mathscr E_4^{}\left( \frac{1+\sqrt{7}i}{2} \right)+2\sqrt{7}i\mathscr E_4'\left( \frac{1+\sqrt{7}i}{2} \right)=\frac{7\sqrt{7}i}{3}-\frac{120\zeta(3)}{\pi^{3}i}-\frac{4\sqrt{7}i}{3}E\left( \frac{1+\sqrt{7}i}{2},2 \right)
\end{align} issues from  \eqref{eq:E(z,2)LR'}, we arrive at  \eqref{eq:EichlerE4''isqrt7} after recalling $ E\left( \frac{1+\sqrt{7}i}{2},2 \right)=\frac{105L_{-7}(2)}{4\pi^2}$ (cf.\ Table \ref{tab:H2}).

To show \eqref{eq:T46isqrt7}, simply study the $ z=\frac{\sqrt{7}i}{2}$ case of\begin{align}
\mathscr T_{1/46}(z)=\frac{\mathscr Q_1(z)-\frac{1}{46}\mathscr Q_2(z)}{\pi(\I z)^2}-\frac{\pi i}{46\I z}\left[3\mathscr E''_4\left( z+\frac{1}{2} \right)-\mathscr E''_4(2z)\right]
\end{align}along with the proven identity  $ \mathscr  Q_1\left( \frac{1+\sqrt{7}i}{2} \right)-\frac{1}{46} \mathscr  Q_2\left( \frac{1+\sqrt{7}i}{2} \right)=\frac{43\pi^{2}}{552}-\frac{245L_{-7}(2)}{368}$  (cf.\ Table \ref{tab:H2}).

\item With  $ z=\frac{i}{\sqrt{2}}$ in \eqref{eq:EichlerE4''refl}, we are left with\begin{align}
\mathscr E_4''\left( \frac{i}{\sqrt{2}} \right)+2\mathscr E_4''\left( \sqrt{2} i\right)-2\mathscr E_4^{}\left( \sqrt{2} i\right)+2\sqrt{2}i\mathscr E_4'\left( \sqrt{2} i\right)=-\frac{7\sqrt{2}i}{3}-\frac{60\zeta(3)}{\pi^{3}i},
\end{align}while \begin{align}
-2\mathscr E_4^{}\left( \sqrt{2} i\right)+2\sqrt{2}i\mathscr E_4'\left( \sqrt{2} i\right)=\frac{8\sqrt{2}i}{3}-\frac{60\zeta(3)}{\pi^{3}i}+\frac{4\sqrt{2}E\big(\sqrt{2}i,2\big)}{3i}
\end{align}is a result of \eqref{eq:E(z,2)LR'}. Therefore,  we may verify \eqref{eq:EichlerE4''isqrt2} through   $ E\big(\sqrt{2}i,2\big)=E\big(1+\sqrt{2}i,2\big)=\frac{30L_{-8}(2)}{\pi^2}$ (cf.\ Table \ref{tab:H2}).

The formula\begin{align}\mathscr T_{1/4}(z)=
\frac{\mathscr Q_1(z)-\frac{1}{4}\mathscr Q_2(z)}{\pi(\I z)^2}-\frac{\pi i}{20\I z}\left[\mathscr E''_4\left( z+\frac{1}{2} \right)+2\mathscr E''_4(2z)\right]
\end{align} becomes   \eqref{eq:T4isqrt2} when $ z=\frac{1}{2}+\frac{i}{\sqrt{2}}$, thanks to the proven evaluation  $ \mathscr  Q_1\left( \frac{1}{2}+\frac{i}{\sqrt{2}} \right)-\frac{1}{4} \mathscr  Q_2\left( \frac{1}{2}+\frac{i}{\sqrt{2}} \right)=\frac{5\pi^{2}}{24}-2L_{-8}(2)$ (cf.\ Table \ref{tab:H2}).
\item Setting $ z=i$ in \eqref{eq:EichlerE4''refl}, we find\begin{align}
2\mathscr E''_4(i)-2\mathscr E_4^{}(i)+2i\mathscr E_4'(i)=-\frac{8i}{3}-\frac{60\zeta(3)}{\pi^{3}i},
\end{align}while \begin{align}
-2\mathscr E_4^{}(i)+2i\mathscr E_4'(i)=\frac{4i}{3}-\frac{60\zeta(3)}{\pi^{3}i}+\frac{4E(i,2)}{3i}
\end{align}follows from \eqref{eq:E(z,2)LR'}. Therefore,  we may confirm \eqref{eq:EichlerE4''i} right after invoking  $ E(i,2)=\frac{30L_{-4}(2)}{\pi^2}$ from  Table \ref{tab:H2}.

Specializing  \eqref{eq:T16} to $ z=\frac{1}{2}+i$, while looking up  the entry $ \mathscr  Q_1\left( \frac{1}{2}+i \right)-\frac{1}{16} \mathscr  Q_2\left( \frac{1}{2}+i \right)=\frac{11\pi^{2}}{96} -\frac{5L_{-4}(2)}{4}$ in the penultimate column of Table \ref{tab:H2}, we reach \eqref{eq:T16i}. \qedhere\end{enumerate}
\end{proof}

\begin{remark}It might be a little counterintuitive that the generalizations of  \eqref{eq:-64H2}--\eqref{eq:4096H2} [see Theorem \ref{thm:H2}(b)] are more sophisticated than those of  \eqref{eq:Sun1}--\eqref{eq:Sun4} [see Theorem \ref{thm:H2}(a)]. As shown in the proof above, the ``magic cancellations'' of all the special Dirichlet  $L$-values in the final forms of \eqref{eq:-64H2}--\eqref{eq:4096H2} may have obscured the actual complexity of Theorem \ref{thm:H2}(b).
\eor\end{remark}
\section{Infinite series involving $ \binom{2k}k^3$, $ \mathsf H_k^{(3)}$, and $ \mathsf H_{2k}^{(3)}$\label{sec:H3}}In \S\ref{subsec:3deriv}, we upgrade the procedures of \S\ref{subsec:pFqODE} to third-order partial derivatives with respect to the  hypergeometric parameters, as well  the  integral representations that solve the corresponding inhomogeneous differential equations.
In \S\ref{subsec:Thm1.2}, we prove the main statements of Theorem  \ref{thm:H3}, by introducing modular parametrizations to the integral representations constructed in \S\ref{subsec:3deriv}. In \S\ref{subsec:specEichlerE6}, we study the special cases listed in Table \ref{tab:H3}.
  \subsection{Third-order hypergeometric deformations and variation of parameters\label{subsec:3deriv}}We open this subsection by extending the first half of Lemma \ref{lm:var_para}.   \begin{lemma}\label{lm:H3a}For   $ |4t(1-t)|\leq1$ and $ \RE t\leq\frac12$, we have  {\allowdisplaybreaks
\begin{align}
\begin{split}{}&\sum_{k=0}^\infty\binom{2k}{k}^3\left[ \mathsf H^{(3)}_{2k}-\frac{1}{8}\mathsf H^{(3)}_k \right]\left[ \frac{t(1-t)}{2^4} \right]^{k}=-\frac{1}{96}\left.\!\frac{\partial^3}{\partial\varepsilon^3}{_3F_2}\left(\left. \begin{array}{@{}c@{}}
 \frac{1}{2}+\varepsilon,\frac{1}{2}+\varepsilon,\frac{1}{2}-2\varepsilon \\
1,1 \\
\end{array} \right| 4t(1-t)\right)\right|_{\varepsilon=0}\\={}&\left(\frac{2}{\pi}\right)^4\int_0^t(1-2s)\big[\mathbf K\big(\sqrt{s}\big)\big]^2\big[\mathbf K\big(\sqrt{\mathstrut 1-s}\big)\mathbf K\big(\sqrt{\mathstrut t}\big)-\mathbf K\big(\sqrt{\mathstrut s}\big)\mathbf K\big(\sqrt{\mathstrut 1-t}\big)\big]^{2}\D s.\end{split}\label{eq:H3int1}
\end{align}
}
\end{lemma} \begin{proof}The first equality in \eqref{eq:H3int1} follows immediately  from termwise differentiation of the $ _3F_2$ series.

With the Appell--Legendre differential operator\footnote{The operator $ \widehat {\mathscr A}$ was written as $ \widehat{\mathsf A}_{-1/2,t}$ in \cite[Proposition 4.8]{Zhou2025NotesBHS}, where the subscript $ -1/2$ referred to the degree of the Legendre function.}\begin{align}
\widehat{\mathscr A}f(t)\colonequals\frac{\partial }{\partial t}\left[t (1-t)\frac{\partial }{\partial t} \right]^2f(t)-\sqrt{t(1-t)}\frac{\partial}{\partial t}\left[ \sqrt{t(1-t)}f(t) \right],
\end{align} we may construct the following differential equation:\begin{align}\left[
\widehat{\mathscr A}+12\varepsilon^{2}t(1-t)\frac{\partial}{\partial t}+2\varepsilon^{2}(3+4\varepsilon)(1-2t)\right]{_3F_2}\left(\left. \begin{array}{@{}c@{}}
 \frac{1}{2}+\varepsilon,\frac{1}{2}+\varepsilon,\frac{1}{2}-2\varepsilon \\
1,1 \\
\end{array} \right| 4t(1-t)\right)=0,
\end{align}whose third-order derivative in $ \varepsilon$ steers us to an  inhomogeneous Appell--Legendre differential equation, namely\begin{align}\begin{split}\widehat{\mathscr A}
\left.\!\frac{\partial^3}{\partial\varepsilon^3}{_3F_2}\left(\left. \begin{array}{@{}c@{}}
 \frac{1}{2}+\varepsilon,\frac{1}{2}+\varepsilon,\frac{1}{2}-2\varepsilon \\
1,1 \\
\end{array} \right| 4t(1-t)\right)\right|_{\varepsilon=0}={}&-48(1-2t){_3F_2}\left(\left. \begin{array}{@{}c@{}}
 \frac{1}{2},\frac{1}{2},\frac{1}{2} \\
1,1 \\
\end{array} \right| 4t(1-t)\right)\\\xlongequal{\text{\eqref{eq:KK}}}{}&-48(1-2t)\left( \frac{2}{\pi} \right)^2\big[\mathbf K\big(\sqrt{t}\big)\big]^2.\label{eq:AL1}
\end{split}\end{align}In doing so, we have already exploited the fact that \begin{align}
\left.\!\frac{\partial}{\partial\varepsilon}{_3F_2}\left(\left. \begin{array}{@{}c@{}}
 \frac{1}{2}+\varepsilon,\frac{1}{2}+\varepsilon,\frac{1}{2}-2\varepsilon \\
1,1 \\
\end{array} \right| 4t(1-t)\right)\right|_{\varepsilon=0}=0
\end{align}for $ |4t(1-t)|\leq1$, which in turn, is a  consequence of termwise differentiation.

Solving the inhomogeneous equation  \eqref{eq:AL1} by variation of parameters
(see \cite[Proposition 4.8]{Zhou2025NotesBHS} for the treatment of a similar-looking one), we arrive at the integral representation in \eqref{eq:H3int1}, up to a trailing term that is a linear combination of $ \big[\mathbf K\big(\sqrt{t}\big)\big]^2$, $ \mathbf K\big(\sqrt{t}\big)\mathbf K\big(\sqrt{1-t}\big)$, and $ \big[\mathbf K\big(\sqrt{1-t}\big)\big]^2$. Since the left-hand side of  \eqref{eq:H3int1} has order $ O(t)$ in the $ t\to0^+$ regime, while\begin{align}
\mathbf K\big(\sqrt{t}\big)=\frac{\pi}{2}+O(t),\quad \mathbf K\big(\sqrt{1-t}\big)=\frac{1}{2}\log\frac{16}{t}+O(t\log t),
\end{align}  such a linear combination must vanish identically.\end{proof}To extend the second half of Lemma \ref{lm:var_para} to third-order harmonic numbers, we require an additional twist, as explained in the next lemma. \begin{lemma}\label{lm:H3b}For   $ |4t(1-t)|\leq1$ and $ \RE t\leq\frac12$, we have \begin{align}
\begin{split}{}&\sum_{k=0}^\infty\binom{2k}{k}^3\mathsf H^{(3)}_k \left[ \frac{t(1-t)}{2^4} \right]^{k}=\frac{1}{12}\left.\!\frac{\partial^3}{\partial\varepsilon^3}{_4F_3}\left(\left. \begin{array}{@{}c@{}}
 \frac{1}{2},\frac{1}{2},\frac{1}{2},1 \\
1+\varepsilon,1+\varepsilon,1-2\varepsilon\ \\
\end{array} \right| 4t(1-t)\right)\right|_{\varepsilon=0}\\={}&\left(\frac{2}{\pi}\right)^4\int_0^t\frac{2(1-2s)}{s(1-s)}\left\{\big[\mathbf K\big(\sqrt{s}\big)\big]^2-\left(\frac{\pi}{2}\right)^{2}\right\}\big[\mathbf K\big(\sqrt{\mathstrut 1-s}\big)\mathbf K\big(\sqrt{\mathstrut t}\big)-\mathbf K\big(\sqrt{\mathstrut s}\big)\mathbf K\big(\sqrt{\mathstrut 1-t}\big)\big]^{2}\D s.\end{split}\label{eq:H3int2}
\end{align}
\end{lemma}\begin{proof}The first equality in \eqref{eq:H3int2}  is an elementary exercise in termwise differentiation. The rationale behind the second equality will occupy the rest of this proof, in four separate paragraphs.

Firstly, with\begin{align}
\left\{\begin{array}{@{}r@{{}\colonequals{}}l}
f_{\nu,\varepsilon}(T) &
\sqrt{\frac{\pi}{-T}} \frac{2 \tan (\nu\pi   ) \Gamma (1-2\varepsilon ) [\Gamma (1+\varepsilon)]^2}{(1+2 \nu ) \Gamma \left(\frac{1}{2}-2\varepsilon\right) \left[\Gamma \left(\varepsilon+\frac{1}{2}\right)\right]^2}
{_3F_2}\left(\!\begin{smallmatrix}\frac{1}{2}-\varepsilon,\frac{1}{2}-\varepsilon,\frac{1}{2}+2\varepsilon  \\\frac{1}{2}-\nu ,\frac{3}{2}+\nu \end{smallmatrix}\!\middle|\frac{1}{T}\right)  \\[8pt]
g_{\nu,\varepsilon}(T) & \frac{(-T)^\nu}{\sqrt{\pi } }\frac{\Gamma \left(\frac{1}{2}+\nu \right) \Gamma (1+2 \nu)\Gamma (1-2\varepsilon )[ \Gamma (1+\varepsilon)]^2 }{\Gamma (1+\nu -2 \varepsilon ) [\Gamma (1 +\nu+\varepsilon)]^2} {_3F_2}\left(\!\begin{smallmatrix}-\nu -\varepsilon,-\nu -\varepsilon ,-\nu+2 \varepsilon  \\\frac{1}{2}-\nu ,-2 \nu \end{smallmatrix}\!\middle|\frac{1}{T}\right) \\[8pt]
h_{\nu,\varepsilon}(T) & \frac{(-T)^{-1-\nu }}{\sqrt{\pi}}\frac{\Gamma (-1-2 \nu ) \Gamma \left( -\frac{1}{2}-\nu\right)  \Gamma (1-2\varepsilon )[ \Gamma (1+\varepsilon)]^2}{\sqrt{\pi }  \Gamma (-\nu-2\varepsilon  ) [\Gamma (-\nu +\varepsilon )]^2}{_3F_2}\left(\!\begin{smallmatrix}1+\nu -\varepsilon,1+\nu -\varepsilon,1+\nu +2\varepsilon  \\\frac{3}{2}+\nu ,2+2 \nu  \end{smallmatrix}\!\middle|\frac{1}{T}\right) \\
\end{array}\right.\label{eq:fgh}
\end{align}defined for $ 2\nu\notin\mathbb Z$ and $ \varepsilon\in\big(0,\frac12\big)$, we observe that\begin{align}
\begin{split}0={}&{_4F_3}\left(\left. \begin{array}{@{}c@{}}
 -\nu,1+\nu,\frac{1}{2},1 \\
1+\varepsilon,1+\varepsilon,1-2\varepsilon\ \\
\end{array} \right| T\right)-\frac{4\varepsilon^{3}}{\nu (\nu +1)_{}T}{_4F_3}\left(\left. \begin{array}{@{}c@{}}
 1,1-\varepsilon,1-\varepsilon,1+2\varepsilon \\
\frac{3}{2},1-\nu,2+\nu\ \\
\end{array} \right| \frac{1}{T}\right)\\&{}-f_{\nu,\varepsilon}(T)-g_{\nu,\varepsilon}(T)-h_{\nu,\varepsilon}(T)\end{split}
\end{align}  holds for $|T|=1$ and $ T\neq1$, because the right-hand side of the equation above represents the  sum over all the residues of\begin{align}
 -\frac{\sin(\nu\pi)}{\pi}\frac{ \Gamma (1-2 s)  \Gamma (s)\Gamma (-\nu -s) \Gamma (1+\nu-s) \Gamma (1-2\varepsilon) [\Gamma (1+\varepsilon)]^2}{   \Gamma (1-s-2\varepsilon) [\Gamma (1-s+\varepsilon)]^2}\frac{2^{2 s}}{(-T)^s}
\end{align}  in the complex $s$-plane. (This is a variation on  the    Guillera--Rogers trick \cite[Lemma 1]{GuilleraRogers2014}.)

Secondly, it so happens that  the operator\begin{align}
\widehat{\mathscr A}- \left[ 3\varepsilon ^2-(1+2 \nu )^2 t(1-t)  \right]\frac{\partial}{\partial t}-(1-2t)\left[ \frac{2 \varepsilon ^3}{t (1-t)}-\frac{(1+2 \nu )^2}{2}  \right]
\end{align}simultaneously annihilates $ f_{\nu,\varepsilon}(4t(1-t)$, $g_{\nu,\varepsilon}(4t(1-t)$, and $h_{\nu,\varepsilon}(4t(1-t)$, thus the function\begin{align}
\begin{split}\varphi_\varepsilon(T)\colonequals {}&\lim_{\nu\to-1/2}\left[ f_{\nu,\varepsilon}(T)+g_{\nu,\varepsilon}(T)+h_{\nu,\varepsilon}(T) \right]\\={}&{_4F_3}\left(\left. \begin{array}{@{}c@{}}
 \frac{1}{2},\frac{1}{2},\frac{1}{2},1 \\
1+\varepsilon,1+\varepsilon,1-2\varepsilon\ \\
\end{array} \right| T\right)+\frac{16\varepsilon^{3}}{T}{_4F_3}\left(\left. \begin{array}{@{}c@{}}
 1,1-\varepsilon,1-\varepsilon,1+2\varepsilon \\
\frac{3}{2},\frac{3}{2},\frac{3}{2} \\
\end{array} \right| \frac{1}{T}\right)\label{eq:phi_eps_refl}
\end{split}
\end{align}satisfies
a homogeneous differential equation\begin{align}
\left[ \widehat{\mathscr A}-3\varepsilon^{2}\frac{\partial}{\partial t}-\frac{2\varepsilon^{3}(1-2t)}{t(1-t)}\right]\varphi_\varepsilon(4t(1-t))=0,
\end{align}and the relation \begin{align}
\left.\!\frac{\partial\varphi_\varepsilon(4t(1-t))}{\partial\varepsilon}\right|_{\varepsilon=0}\xlongequal{\text{\eqref{eq:phi_eps_refl}}}
\left.\!\frac{\partial}{\partial\varepsilon}{_4F_3}\left(\left. \begin{array}{@{}c@{}}
 \frac{1}{2},\frac{1}{2},\frac{1}{2},1 \\
1+\varepsilon,1+\varepsilon,1-2\varepsilon\ \\
\end{array} \right| 4t(1-t)\right)\right|_{\varepsilon=0}=0
\end{align}follows from termwise differentiation of the $ _4F_3$ series.

Thirdly, solving
\begin{align}
\begin{split}\widehat{\mathscr A}\left.\!\frac{\partial^{3}\varphi_\varepsilon(4t(1-t))}{\partial\varepsilon^{3}}\right|_{\varepsilon=0}={}&\frac{12(1-2t)}{t(1-t)}\varphi_0(4t(1-t))\underset{\eqref{eq:phi_eps_refl}}{\xlongequal{\text{\eqref{eq:KK}}}}\frac{12(1-2t)}{t(1-t)}\left( \frac{2}{\pi} \right)^2\big[\mathbf K\big(\sqrt{t}\big)\big]^2
\end{split}
\end{align}
by variation of parameters, we get\begin{align}
\begin{split}&\left.\!\frac{\partial^{3}\varphi_\varepsilon(4t(1-t))}{\partial\varepsilon^{3}}\right|_{\varepsilon=0}\\={}&-\left(\frac{2}{\pi}\right)^4\int_t^\infty\frac{24(1-2s)}{s(1-s)}\big[\mathbf K\big(\sqrt{s}\big)\big]^2\big[\mathbf K\big(\sqrt{\mathstrut 1-s}\big)\mathbf K\big(\sqrt{\mathstrut t}\big)-\mathbf K\big(\sqrt{\mathstrut s}\big)\mathbf K\big(\sqrt{\mathstrut 1-t}\big)\big]^{2}\D s\\&{}+\tilde c_1\big[\mathbf K\big(\sqrt{t}\big)\big]^2+\tilde c_2 \mathbf K\big(\sqrt{t}\big)\mathbf K\big(\sqrt{1-t}\big)+\tilde c_3 \big[\mathbf K\big(\sqrt{1-t}\big)\big]^2,
\end{split}
\end{align}where
the exact values of the constants $ \tilde c_1$, $\tilde c_2$, and $ \tilde c_3$ are irrelevant at this moment.

Lastly, recall from \cite[Proposition 4.8]{Zhou2025NotesBHS} that
\begin{align}
\begin{split}{}&\frac{24}{t(1-t)}{_4F_3}\left(\left. \begin{array}{@{}c@{}}
 1,1,1,1 \\
\frac{3}{2},\frac{3}{2},\frac{3}{2} \\
\end{array} \right| \frac{1}{4t(1-t)}\right)\\={}&-\left(\frac{2}{\pi}\right)^2\int_t^\infty\frac{24(1-2s)}{s(1-s)}\big[\mathbf K\big(\sqrt{\mathstrut 1-s}\big)\mathbf K\big(\sqrt{\mathstrut t}\big)-\mathbf K\big(\sqrt{\mathstrut s}\big)\mathbf K\big(\sqrt{\mathstrut 1-t}\big)\big]^{2}\D s
\end{split}\label{eq:4F3Rama}
\end{align}holds for $ |4t(1-t)|\geq1$, so the third-order derivative of \eqref{eq:phi_eps_refl} allows us to confirm the second equality in \eqref{eq:H3int2} for  $| 4t(1-t)|=1$---there are no non-trivial linear combinations of $ \big[\mathbf K\big(\sqrt{t}\big)\big]^2$, $ \mathbf K\big(\sqrt{t}\big)\mathbf K\big(\sqrt{1-t}\big)$, and $ \big[\mathbf K\big(\sqrt{1-t}\big)\big]^2$ in the  trailing term,  due to the  $ O(t)$ behavior of  \eqref{eq:H3int2} when analytically continued to the $ t\to0^+$ regime.   \end{proof}
\subsection{Representations through Eichler integrals and Epstein zeta functions\label{subsec:Thm1.2}}
For $ z/i>0$, we have [cf.\ \eqref{eq:zKratio}]\begin{align}
2=\frac{\partial}{\partial z}\frac{iP_{-1/2}(2\alpha_4(z)-1)}{P_{-1/2}(1-2\alpha_{4}(z))}=\frac{1}{\pi i\alpha_4(z)[1-\alpha_4(z)][P_{-1/2}(1-2\alpha_{4}(z))]^{2}}\frac{\partial\alpha_4(z)}{\partial z},\label{eq:alpha4'}
\end{align}along with Ramanujan's relations for the Eisenstein series \cite[pp.~126--127]{RN3}:\begin{align}
\left\{\begin{array}{@{}r@{{}=\left[P_{-1/2}(1-2\alpha_{4}(z))\right]^4}l}
E_{4}(z) & \left\{1+14\alpha_4(z)+[\alpha_{4}(z)]^{2}\right\},\\
E_{4}(2z) & \left\{1-\alpha_4(z)+[\alpha_{4}(z)]^{2}\right\}, \\
E_{4}(4z) & \left\{1-\alpha_4(z)+\frac{[\alpha_{4}(z)]^{2}}{16}\right\}, \\
\end{array}\right.
\end{align}and \begin{align}
\left\{\begin{array}{@{}r@{{}=\left[P_{-1/2}(1-2\alpha_{4}(z))\right]^6}l}
E_{6}(z) & [1+\alpha_{4}(z)]\left\{1-34\alpha_4(z)+[\alpha_{4}(z)]^{2}\right\},\\[3pt]
E_{6}(2z) & [1+\alpha_{4}(z)]\left[ 1-\frac{\alpha_{4}(z)}{2} \right][1-2\alpha_{4}(z)], \\[3pt]
E_{6}(4z) & \left[ 1-\frac{\alpha_{4}(z)}{2} \right]\left\{1-\alpha_4(z)+\frac{[\alpha_{4}(z)]^{2}}{32}\right\}, \\
\end{array}\right.\label{eq:E6Rama}
\end{align}where $ \alpha_{4}(z)=\lambda(2z)$ and  $ P_{-1/2}(1-2t)=\frac{2}{\pi}\mathbf K\big(\sqrt{t}\big)$. The same  identities admit analytic continuations to regions free from the obstructions by  the branch cut of $ \mathbf K\big(\sqrt{t}\big),t\in\mathbb C\smallsetminus[1,\infty)$.

The formulae listed in the last paragraph will become instrumental in the transition from  the integral representations in \S\ref{subsec:3deriv} to the Eichler integrals in Theorem \ref{thm:H3}. \begin{proof}[Proof of Theorem \ref{thm:H3}]\begin{enumerate}[leftmargin=*,  label=(\alph*),ref=(\alph*),
widest=d, align=left] \item
In parallel to \eqref{eq:sumE4}, we have  \begin{align}
0=E_6\left( z+\frac{1}{2} \right)+E_6(z)-66E_{6}(2z)+64E_6(4z),\label{eq:sumE6}
\end{align}so  \eqref{eq:E6Rama} can be reassembled into\begin{align}
E_6\left( z+\frac{1}{2} \right)-E_{6}(z)=\frac{63\left[P_{-1/2}(1-2\alpha_{4}(z))\right]^6[1-2\alpha_{4}(z)]\alpha_4(z)[1-\alpha_4(z)]}{2}.
\end{align}Combining the last displayed equation with \eqref{eq:zKratio} and \eqref{eq:alpha4'}, we find\begin{align}
\begin{split}&\int_0^{\lambda(2z)}(1-2s)\big[\mathbf K\big(\sqrt{s}\big)\big]^4\left[\frac{i\mathbf K\big(\sqrt{\mathstrut 1-s}\big)}{\mathbf K\big(\sqrt{\mathstrut s}\big)} -2z\right]^{2}\D s\\={}&\frac{\pi^5 i}{63}\int_{z+\frac{1}{2}}^{i\infty}[E_6(2w)-E_6(w)]\left(w-z-\frac{1}{2}\right)^{2}\D w=\frac{\pi^{5}i\big[8\mathscr E''_6\big(z+\frac12\big)-\mathscr E''_6(2z)\big]}{6048},
\end{split}
\end{align} which enables us to transcribe     \eqref{eq:H3int1}  into \eqref{eq:H3Q1}.

Essentially the same workflow brings us\begin{align}
\begin{split}&\frac{E_{4}(z)-16E_4(4z)}{15}-\frac{E_6\big( z+\frac{1}{2} \big)-64E_{6}(2z)}{63}\\={}&\left[P_{-1/2}(1-2\alpha_{4}(z))\right]^4\left\{ \left[P_{-1/2}(1-2\alpha_{4}(z))\right]^2-1 \right\}[1-2\alpha_{4}(z)]
\end{split}
\end{align}and\begin{align}
\begin{split}&\int_0^{\lambda(2z)}\frac{2(1-2s)}{s(1-s)}\big[\mathbf K\big(\sqrt{s}\big)\big]^2\left\{ \big[\mathbf K\big(\sqrt{s}\big)\big]^2 -\left(\frac{\pi}{2}\right)^{2}\right\}\left[\frac{i\mathbf K\big(\sqrt{\mathstrut 1-s}\big)}{\mathbf K\big(\sqrt{\mathstrut s}\big)} -2z\right]^{2}\D s\\={}&-\pi^5 i\int_{z}^{i\infty}\left[ \frac{E_{4}(w)-16E_4(4w)}{15}-\frac{E_6\big( w+\frac{1}{2} \big)-64E_{6}(2w)}{63} \right]\left(w-z\right)^{2}\D w\\={}&\frac{\pi^5 i\big[4\mathscr E_4(z)-\mathscr E_4(4z)\big]}{60}-\frac{\pi^{5}i\big[\mathscr E''_6\big(z+\frac12\big)-8\mathscr E''_6(2z)\big]}{756},
\end{split}
\end{align} which transform \eqref{eq:H3int2} into  \eqref{eq:H3Q2}.\item Divide both sides of \eqref{eq:H3Q1} and   \eqref{eq:H3Q2} by $ -\pi\I z$, before differentiating in $ \I z$. One may refer to  \eqref{eq:E(z,2)LR'} for an integral representation of the Epstein zeta function.  \qedhere\end{enumerate}
\end{proof}\begin{remark}We may convert \cite[(1.8), $ m=2$]{EZF}
 into the following form:\begin{align}
E(z,3)={}&(\I z)^3+\frac{2835 \zeta (5)}{8 \pi ^5(\I z)^2}-\frac{15}{8(\I z)^2}\RE\left\{ i\int_z^{i\infty}[1-E_6(w)](w-z)^{2}(w-\overline z)^{2}\D w \right\},\label{eq:E(z,3)LR'}
\end{align}where $ z\in\mathfrak H$.
[If $ 2\RE z\in\mathbb Z$, then the expression inside the braces of \eqref{eq:E(z,3)LR'} is real-valued.] Here, a polynomial identity $ i(w-z)^{2}(w-\overline z)^{2}=i(w-z)^{4}-4(w-z)^3\I z-4i(w-z)^{2}(\I z)^2$ allows us to put down\begin{align}
i\int_z^{i\infty}[1-E_6(w)](w-z)^{2}(w-\overline z)^{2}\D w=i\mathscr E_6^{}(z)+ \mathscr E'_6(z)\I z-\frac{i \mathscr E''_6(z)(\I z)^2}{3},\label{eq:braced}
\end{align} where \begin{align}
\mathscr E_6'(z)\colonequals \frac{\partial\mathscr E_6(z)}{\partial z}=4\int_{z}^{i\infty}[1-E_6(w)](z-w)^3\D w.
\end{align}In principle, one may reformulate  Theorem \ref{thm:H3} into a form similar to that of Theorem \ref{thm:H2}, getting rid of $ \mathscr E''_6$  while keeping only $ E(\cdot,3)$, $ \mathscr E_6^{}$, $ \mathscr E'_6$, and $ \mathscr E'''_6$. In practice, such a  reformulation is more cumbersome than helpful: as will be indicated at the end of \S\ref{subsec:specEichlerE6}, no  closed-form evaluations [in the spirit of \eqref{eq:Sun1}--\eqref{eq:Sun4}] will emerge even after one combines  \eqref{eq:E(z,3)LR'} with Theorem \ref{thm:H3}(a).
\eor\end{remark}
\subsection{Special  sum rules for $ \mathscr E_6$ and derivatives\label{subsec:specEichlerE6}}
Before delving into special values of $ \check {\mathscr T}_{\check r(z)}(z)$ \big[defined by \eqref{eq:Ur(z)}, via $ \mathscr E''_6$ and $ \mathscr E'''_6$\big] in the last column of Table \ref{tab:H3}, we give an overview of some  properties enjoyed by the function  $ \mathscr E_6(z)$ for generic $z\in\mathfrak H $. Note that the Eichler integral $ \mathscr E_6$ has a  Lambert--Ramanujan  avatar:\begin{align}
\mathscr E_6(z)=\frac{378i}{\pi^5}\sum_{n=1}^\infty\frac{1}{n^5(e^{-2\pi in z}-1)}.\tag{\ref{eq:EichlerE6defn}$'$}\label{eq:EichlerE6defn'}
\end{align}Thus, for $ m=2$, the Lambert--Ramanujan reflection formula \eqref{eq:Ramanujan_reflection_notebook} can be rewritten as \begin{align}
\mathscr E_6(z)-z^4\mathscr E_6\left( -\frac{1}{z} \right)=-\frac{\left(z^2+1\right) \left(2 z^4-9 z^2+2\right)}{10 z}-\frac{189\zeta (5)\left(z^4-1\right) }{\pi ^5i}.\label{eq:EichlerE6refl}
\end{align}Thanks to \eqref{eq:EichlerE6defn'} and \eqref{eq:E(z,3)LR'}, the sum rule \eqref{eq:sumE6} has two more  companions:\begin{align}
0=16\mathscr E_6\left( z+\frac{1}{2} \right)+16\mathscr E_6(z)-33\mathscr E_{6}(2z)+\mathscr E_6(4z)\label{eq:sumEichlerE6}
\end{align}and\begin{align}
0=4E\left( z+\frac{1}{2} ,3\right)+4E(z,3)-33E(2z,3)+4E(4z,3),\label{eq:E(z,3)add}
\end{align}both of which are valid  for all $ z\in\mathfrak H$.

The aforementioned identities will facilitate the computations of   $ \mathscr E_6(z)$ and its derivatives at special points $z$, which are indispensable to the verification of \eqref{eq:SunH3a}--\eqref{eq:SunH3d} in the next proposition.\begin{proposition}\label{prop:H3proof}\begin{enumerate}[leftmargin=*,  label=\emph{(\alph*)},ref=(\alph*),
widest=d, align=left] \item
We have \begin{align}
i\mathscr E''_6\left( \frac{1+\sqrt{3}i}{2} \right)+\frac{\sqrt{3}}{2}\mathscr E'''_6\left( \frac{1+\sqrt{3}i}{2} \right)=3\sqrt{3}-\frac{168\zeta (3)}{\pi^{3}},\label{eq:EichlerE6isqrt3}
\end{align}which entails\begin{align}
\check {\mathscr T}_{1/64}\left( \frac{\sqrt{3}i}{2} \right)=\frac{25\zeta(3)}{24\pi}.\label{eq:Usqrt3i}
\end{align}
\item We have \begin{align}
\begin{split}{}&2i\left[ 39\mathscr E''_6\left( \frac{1+\sqrt{7}i}{2} \right) -4\mathscr E''_6\left( \sqrt{7}i \right)\right]+\sqrt{7}\left[ 39\mathscr E'''_6\left( \frac{1+\sqrt{7}i}{2} \right) -8\mathscr E'''_6\left( \sqrt{7}i \right)\right]\\={}&98 \sqrt{7}-\frac{6912 \zeta (3)}{\pi ^3},
\end{split}\label{eq:EichlerE6isqrt7}
\end{align}which entails\begin{align}
\check {\mathscr T}_{1/352}\left( \frac{\sqrt{7}i}{2} \right)=\frac{555 \zeta (3)}{2156 \pi }.\label{eq:Uisqrt7}
\end{align}
\item
We have \begin{align}
i\mathscr E''_6\left( \frac{i}{\sqrt{2}} \right)+i\mathscr E''_6\left( \sqrt{2}i \right)+\frac{1}{\sqrt{2}}\mathscr E'''_6\left( \frac{i}{\sqrt{2}} \right)+\sqrt{2}\mathscr E'''_6\left( \sqrt{2}i \right)=18 \sqrt{2}-\frac{567 \zeta (3)}{\pi ^3},\label{eq:EichlerE6isqrt2}
\end{align}which entails\begin{align}
\check {\mathscr T}_{1/8}\left(\frac{1}{2}+ \frac{{}i}{\sqrt{2}} \right)=\frac{15\zeta(3)}{4\pi}.
\label{eq:Uisqrt2}\end{align}

\item We have \begin{align}
i\mathscr E''_6(i)+\mathscr E'''_6(i)=8-\frac{189\zeta (3)}{\pi^{3}},\label{eq:EichlerE6i}
\end{align}which entails\begin{align}
\check {\mathscr T}_{1/64}\left( \frac{1}{2} +i\right)=\frac{57 \zeta (3)}{64 \pi }.\label{eq:Ui}
\end{align}

\end{enumerate}
\end{proposition}\begin{proof}\begin{enumerate}[leftmargin=*,  label=(\alph*),ref=(\alph*),
widest=d, align=left] \item
Representing  $ E\left(\frac{1+\sqrt{3}i}{2},3\right)=\frac{105\zeta(3)}{2\pi^3}$ through \eqref{eq:E(z,3)LR'}, while bearing in mind the relation in \eqref{eq:braced}, we may establish \begin{align}
i\mathscr E_6^{}\left( \frac{1+\sqrt{3}i}{2} \right)+\frac{\sqrt{3}}{2}\mathscr E_6'\left( \frac{1+\sqrt{3}i}{2} \right)-\frac{i}{4}\mathscr E_6''\left( \frac{1+\sqrt{3}i}{2} \right)=\frac{189 \zeta (5)}{\pi ^5}-\frac{21 \zeta (3)}{\pi ^3}+\frac{3 \sqrt{3}}{20}.
\end{align} Since $ \mathscr E_6(z)=\mathscr E_6(z+1)$ for all $ z\in\mathfrak H$, the zeroth-, first- and  third-order derivatives of \eqref{eq:EichlerE6refl} at $ z=\frac{1+\sqrt{3}i}{2}$ simplify to\begin{align}\mathscr E_6^{}\left( \frac{1+\sqrt{3}i}{2} \right)=\frac{189\zeta (5)  }{\pi ^5i}+
\frac{11 \sqrt{3} i}{30},
\end{align} \begin{align}
4\mathscr E_6^{}\left( \frac{1+\sqrt{3}i}{2} \right)+\frac{3-\sqrt{3}i}{2}\mathscr E_6'\left( \frac{1+\sqrt{3}i}{2} \right)=\frac{756\zeta (5)  }{\pi ^5i}+\frac{29 \sqrt{3} i}{20}+\frac{1}{20},
\end{align}and\begin{align}
\begin{split}&-48\mathscr E_6^{}\left( \frac{1+\sqrt{3}i}{2} \right)-18\big(1-\sqrt{3}i\big)\mathscr E_6'\left( \frac{1+\sqrt{3}i}{2} \right)\\&{}+6\big(1+\sqrt{3}i\big)\mathscr E_6''\left( \frac{1+\sqrt{3}i}{2} \right)+\big(3-\sqrt{3}i\big)\mathscr E_6'''\left( \frac{1+\sqrt{3}i}{2} \right)\\={}&-\frac{9072 \zeta (5) }{\pi ^5i}-15  \sqrt{3}i-\frac{33}{5},
\end{split}
\end{align} respectively. Solving these simultaneous equations, we find\begin{align}
\left\{\begin{array}{@{}r@{\left( \frac{1+\sqrt{3}i}{2} \right)={}}l}
\mathscr E_6' & \frac{1}{30}, \\[7pt]
\mathscr E_6'' & \frac{84 \zeta (3)}{\pi ^3 i}+2 \sqrt{3}  i, \\[7pt]
\mathscr E_6''' & 10-\frac{168 \sqrt{3} \zeta (3)}{\pi ^3}, \\
\end{array}\right.
\end{align}thereby verifying \eqref{eq:EichlerE6isqrt3} \textit{a fortiori}.

Studying \begin{align}
\check {\mathscr T}_{1/64}\left( z \right)=-\frac{\pi^2i\mathscr E''_6\big(z+\frac12\big)}{192(\I z)^2}-\frac{\pi^2\mathscr E'''_6\big(z+\frac12\big)}{192\I z}+\frac{\pi^2\I z}{24}-\frac{3\zeta(3)}{32\pi(\I z)^2}\label{eq:U_64}
\end{align}for $ z=\frac{\sqrt{3}i}{2}$, we get \eqref{eq:Usqrt3i}.
\item The special values $E\left(\frac{1+\sqrt{7}i}{2},3\right)=\frac{540 \zeta (3)}{7 \pi ^3} $ and $E\left(\sqrt{7}i,3\right)=\frac{3375 \zeta (3)}{7 \pi ^3} $ (cf.\ Table \ref{tab:H3}) correspond to \begin{align}
i\mathscr E_6^{}\left( \frac{1+\sqrt{7}i}{2} \right)+\frac{\sqrt{7}}{2}\mathscr E_6'\left( \frac{1+\sqrt{7}i}{2} \right)-\frac{7i}{12}\mathscr E_6''\left( \frac{1+\sqrt{7}i}{2} \right)={}&\frac{189 \zeta (5)}{\pi ^5}-\frac{72 \zeta (3)}{\pi ^3}+\frac{49 \sqrt{7}}{60}\label{eq:E6isqrt7(a)}\intertext{and}i\mathscr E_6^{}\left(\sqrt{7}i\right)+\sqrt{7}\mathscr E_6'\left(\sqrt{7}i\right)-\frac{7i}{3}\mathscr E_6''\left(\sqrt{7}i\right)={}&\frac{189 \zeta (5)}{\pi ^5}-\frac{1800 \zeta (3)}{\pi ^3}+\frac{392 \sqrt{7}}{15},
\end{align} Extracting the Taylor coefficients of \eqref{eq:sumEichlerE6}, up to the  $ O\left( \left(z+\frac{2}{1+\sqrt{7}i}\right) ^{3}\right)$ term, we get{\allowdisplaybreaks
\begin{align}
16\mathscr E_6^{}\left( -\frac{2}{-1+\sqrt{7}i} \right)+16\mathscr E_6^{}\left( -\frac{2}{1+\sqrt{7}i} \right)-33\mathscr E_6^{}\left( \frac{1+\sqrt{7}i}{2} \right)+\mathscr E_6^{}\left(\sqrt{7}i\right)={}&0,\\8\mathscr E_6'\left( -\frac{2}{-1+\sqrt{7}i} \right)+8\mathscr E_6'\left( -\frac{2}{1+\sqrt{7}i} \right)-33\mathscr E_6'\left( \frac{1+\sqrt{7}i}{2} \right)+2\mathscr E_6'\left(\sqrt{7}i\right)={}&0,\\4\mathscr E_6''\left( -\frac{2}{-1+\sqrt{7}i} \right)+4\mathscr E_6''\left( -\frac{2}{1+\sqrt{7}i} \right)-33\mathscr E_6''\left( \frac{1+\sqrt{7}i}{2} \right)+4\mathscr E_6''\left(\sqrt{7}i\right)={}&0,\\2\mathscr E_6'''\left( -\frac{2}{-1+\sqrt{7}i} \right)+2\mathscr E_6'''\left( -\frac{2}{1+\sqrt{7}i} \right)-33\mathscr E_6'''\left( \frac{1+\sqrt{7}i}{2} \right)+8\mathscr E_6'''\left(\sqrt{7}i\right)={}&0.\label{eq:Phi3}
\end{align}
}To ease later reference, we denote the expressions on the left-hand sides of the last four equations by $\varPhi_0 $, $ \varPhi_1$, $ \varPhi_2$, and $ \varPhi_3$, respectively. Pick a differential operator \begin{align}
\widehat{\mathscr D}\colonequals2\sqrt{7} \frac{\partial^3}{\partial z^3}+\frac{96}{\sqrt{7}}\frac{\partial}{\partial z}+\frac{384 i}{7}.
\end{align}If we interpret the expression\begin{align}
\begin{split}&\left.\!\widehat{\mathscr D}\left[\mathscr E_6(z)-z^4\mathscr E_6\left( -\frac{1}{z} \right)\right]\right|_{z=-\frac{2}{-1+\sqrt{7}i}}+\left.\!\widehat{\mathscr D}\left[\mathscr E_6(z)-z^4\mathscr E_6\left( -\frac{1}{z} \right)\right]\right|_{z=-\frac{2}{1+\sqrt{7}i}}\\&{}+\frac{72}{7}\left[ i\mathscr E_6^{}\left( \frac{1+\sqrt{7}i}{2} \right)+\frac{\sqrt{7}}{2}\mathscr E_6'\left( \frac{1+\sqrt{7}i}{2} \right)-\frac{7i}{12}\mathscr E_6''\left( \frac{1+\sqrt{7}i}{2} \right) \right]\\&{}+\frac{24}{7}\left[ i\mathscr E_6^{}\left(\sqrt{7}i\right)+\sqrt{7}\mathscr E_6'\left(\sqrt{7}i\right)-\frac{7i}{3}\mathscr E_6''\left(\sqrt{7}i\right) \right]-\frac{24i\varPhi_0}{7} -\frac{12\varPhi_1}{\sqrt{7}}-\sqrt{7}\varPhi_3\end{split}\label{eq:sqrt7bal}
\end{align}literally, without resorting to any functional equations other than $ \mathscr E_6(z)= \mathscr E_6(z+1)$ and its derivatives, then we arrive at the left-hand side of \eqref{eq:EichlerE6isqrt7}. If we simplify \eqref{eq:sqrt7bal}  through the right-hand sides of   \eqref{eq:EichlerE6refl} and \eqref{eq:E6isqrt7(a)}--\eqref{eq:Phi3}, then we end up with the right-hand side of
\eqref{eq:EichlerE6isqrt7}.

To verify \eqref{eq:Uisqrt7}, simply investigate the $ z=\frac{\sqrt{7}i}{2}$ case of \begin{align}
\begin{split}\check{\mathscr T}_{1/352}(z)={}&-\frac{\pi^{2}i\left[39 \mathscr E''_6\big(z+\frac12\big)-4\mathscr E''_6(2z) \right]}{7392(\I z)^2}-\frac{\pi^{2}i\left[39 \mathscr E'''_6\big(z+\frac12\big)-8\mathscr E'''_6(2z) \right]}{7392\I z}\\&{}+\frac{\pi^2\I z}{132}-\frac{3\zeta(3)}{176\pi(\I z)^2}.
\end{split}
\end{align}\item By virtue of  \eqref{eq:E(z,3)LR'} and \eqref{eq:braced}, we may transcribe $ E\left(\frac{i}{\sqrt{2}},3\right)=E\left(\sqrt{2}i,3\right)=\frac{2835 \zeta (3)}{32 \pi ^3}$ into the following pair of identities:\begin{align}
i\mathscr E_6^{}\left(\frac{i}{\sqrt{2}}\right)+\frac{1}{\sqrt{2}}{}\mathscr E_6'\left(\frac{i}{\sqrt{2}}\right)-\frac{i}{6}\mathscr E_6''\left(\frac{i}{\sqrt{2}}\right)={}&\frac{189 \zeta (5)}{\pi ^5}-\frac{189 \zeta (3)}{8 \pi ^3}+\frac{\sqrt{2}}{15},\label{eq:E6isqrt2(a)}\\i\mathscr E_6^{}\left(\sqrt{2}i\right)+{\sqrt{2}}{}\mathscr E_6'\left(\sqrt{2}i\right)-\frac{2i}{3}\mathscr E_6''\left(\sqrt{2}i\right)={}&\frac{189 \zeta (5)}{\pi ^5}-\frac{189 \zeta (3)}{2 \pi ^3}+\frac{32 \sqrt{2}}{15}.\label{eq:E6isqrt2(b)}
\end{align}Evaluating \begin{align}
\begin{split}&{}\left.\!-\left( \frac{1}{2\sqrt{2}}\frac{\partial^{3}}{\partial z^{3}}+\frac{3}{\sqrt{2}}\frac{\partial}{\partial z}+3i \right)\left[\frac{\mathscr E_6(z)}{z^{2}}-z^2\mathscr E_6\left( -\frac{1}{z} \right)\right]\right|_{z=\frac{i}{\sqrt{2}}}
\\&{}+30\left[ i\mathscr E_6^{}\left(\frac{i}{\sqrt{2}}\right)+\frac{1}{\sqrt{2}}{}\mathscr E_6'\left(\frac{i}{\sqrt{2}}\right)-\frac{i}{6}\mathscr E_6''\left(\frac{i}{\sqrt{2}}\right)\right]\\&{}-\frac{3}{2}\left[ i\mathscr E_6^{}\left(\sqrt{2}i\right)+{\sqrt{2}}{}\mathscr E_6'\left(\sqrt{2}i\right)-\frac{2i}{3}\mathscr E_6''\left(\sqrt{2}i\right) \right] \end{split}
\end{align}with the aid of \eqref{eq:EichlerE6refl}, \eqref{eq:E6isqrt2(a)} and \eqref{eq:E6isqrt2(b)}, we recover \eqref{eq:EichlerE6isqrt2}.

The formula \begin{align}
\begin{split}\check{\mathscr T}_{1/8}(z)={}&-\frac{\pi^{2}i\left[ \mathscr E''_6\big(z+\frac12\big)+\mathscr E''_6(2z) \right]}{216(\I z)^2}-\frac{\pi^{2}i\left[ \mathscr E'''_6\big(z+\frac12\big)+2\mathscr E'''_6(2z) \right]}{216\I z}\\&{}+\frac{\pi^2\I z}{3}-\frac{3\zeta(3)}{4\pi(\I z)^2}
\end{split}
\end{align}degenerates to \eqref{eq:Uisqrt2} when $ z=\frac12+\frac{i}{\sqrt{2}}$.

\item Similar to part (a), we  build an equation\begin{align}
i\mathscr E_6^{}(i)+{}\mathscr E_6'(i)-\frac{i}{3}\mathscr E_6''(i)=\frac{189 \zeta (5)}{\pi ^5}-\frac{63 \zeta (3)}{2 \pi ^3}+\frac{8}{15}\label{eq:E6i(a)}
\end{align} on the integral representation of $ E(i,3)=\frac{945\zeta(3)}{16\pi^3}$, while \begin{align}
2i\mathscr E_6^{}(i)+{}\mathscr E_6'(i)=\frac{378 \zeta (5)}{\pi ^5}-\frac{13}{10}\label{eq:E6i(b)}
\end{align}and \begin{align}
12i\mathscr E_6^{}(i)+9{}\mathscr E_6'(i)-3i\mathscr E_6''(i)-\mathscr E_6'''(i)=\frac{2268 \zeta (5)}{\pi ^5}-\frac{87}{10}\label{eq:E6i(c)}
\end{align}follow from the $ O(z-i)$  and   $ O\big((z-i)^3\big)$ terms in the  Taylor expansion of \eqref{eq:EichlerE6refl}.\footnote{The $ O((z-i)^0)$ term leaves us only a trivial identity ``$0=0 $'', while the $ O((z-i)^2)$ term produces the same information as the $ O(z-i)$ term.} Via the linear combination $ 6\times$\eqref{eq:E6i(a)}$ +3\times$\eqref{eq:E6i(b)}$-$\eqref{eq:E6i(c)}, we may deduce \eqref{eq:EichlerE6i} immediately.

Specializing \eqref{eq:U_64} to $ z=\frac12+i$, we confirm \eqref{eq:Ui}.\qedhere\end{enumerate}
\end{proof}\begin{remark}As demonstrated above, one may evaluate  $ \mathscr E''_6\left(\frac{1+\sqrt{3}i}{2}\right)$  without knowing $ \mathscr E'''_6\left(\frac{1+\sqrt{3}i}{2}\right)$ beforehand.  Therefore, we can deduce   the following identity from  Theorem \ref{thm:H3}(a):{
\begin{align}\frac{\displaystyle\sum_{k=0}^\infty\binom{2k}k^3\left[\mathsf H_{2k}^{(3)}-\frac{7}{64}\mathsf H_k^{(3)}\right]\frac{1}{2^{8k}}}{\displaystyle\sum_{k=0}^\infty\binom{2k}k^3\frac{1}{2^{8k}}}=\frac{\pi ^3}{32 \sqrt{3}}-\frac{7 \zeta (3)}{16}-\frac{\pi^{3}i\left[ 4 \mathscr E_4\left( \frac{\sqrt{3}i}{2} \right)-\mathscr E_4\left( 2\sqrt{3}i \right)\right]}{960}.
\end{align}
}This is  by far the closest peer to  \eqref{eq:Sun1}--\eqref{eq:Sun4} that involves third-order harmonic numbers. There are at least two reasons why we cannot produce more  series evaluations in this vein: alongside our  inability to isolate $ \mathscr E''_6(z)$ from various sum rules for generic $z$,   we are also hampered by the difficulty of turning $ 4\mathscr E_4(z)-\mathscr E_4(4z)$  into familiar mathematical constants  for the  special cases of   $z$ listed in Table \ref{tab:H3}.
\eor\end{remark}
\section{Discussions and outlook\label{sec:outlook}}
\subsection{Infinite series involving $ \binom{2k}{k}^2$, $ \mathsf H^{(2)}_{k}$, and $ \mathsf H^{(2)}_{2k}$\label{subsec:sqr}}
Here, we construct a partial analog of Theorem \ref{thm:H2}(a) when we have $ \binom{2k}{k}^2$ in each summand, instead of $ \binom{2k}{k}^3$.  \begin{proposition}If  $z$ satisfies the following constraints:\begin{align}
\I z>0,\;|\RE z|\leq\frac12,\;|\alpha_4(z)|<1,
\end{align}then
 we have {\allowdisplaybreaks
\begin{align}
\begin{split}\frac{\displaystyle \sum_{k=0}^\infty\binom{2k}{k}^2\left[ \mathsf H_{2k}^{(2)}-\frac{1}{4}\mathsf H_{ k}^{(2)} \right]\left[\frac{\alpha_{4}(z)}{2^{4}}\right]^k}{\displaystyle  \sum_{k=0}^\infty\binom{2k}{k}^2\left[\frac{\alpha_{4}(z)}{2^{4}}\right]^k}={}&\sum_{n=0}^\infty\frac{1}{(2n+1)^{2}\cosh^2\frac{(2n+1)\pi z_{}}{i}},\end{split}\label{eq:LR1sqr}\\\begin{split}\frac{\displaystyle \sum_{k=0}^\infty\binom{2k}{k}^2\mathsf H_{ k}^{(2)}\left[\frac{\alpha_{4}(z)}{2^{4}}\right]^k}{\displaystyle  \sum_{k=0}^\infty\binom{2k}{k}^2\left[\frac{\alpha_{4}(z)}{2^{4}}\right]^k}={}&\sum_{n=1}^\infty\frac{2}{n^{2}\cosh\frac{2n\pi z_{}}{i}}-\sum_{n=1}^\infty\frac{1}{n^{2}\cosh^2\frac{2n\pi z_{}}{i}},
\end{split}\label{eq:LR2sqr}
\end{align}
}where \begin{align}
\sum_{k=0}^\infty\binom{2k}{k}^2\left[\frac{\alpha_{4}(z)}{2^{4}}\right]^k=\frac{2}{\pi}\mathbf K\big(\sqrt{\alpha_4(z)}\big)=P_{-1/2}(1-2\alpha_4(z)).
\end{align}
\end{proposition}\begin{proof}The left-hand side of \eqref{eq:LR1sqr} is equal to \begin{align}
-\frac{1}{8P_{-1/2}(1-2\alpha_4(z))}\left.\!\frac{\partial^2 P_\nu(1-2\alpha_4(z))}{\partial\nu^{2}}\right|_{\nu=-1/2},
\end{align}so its right-hand side follows from the proof of  \eqref{eq:LR1}.

It now suffices to prove   \eqref{eq:LR2sqr} for $ z/i>0$, whereupon $t=\alpha_4(z)\in(0,1) $. By differentiation of Euler's beta integral, one can show that \begin{align}
\int_0^1 \frac{x^k\log^{r-1}x}{1-x}\D x=\mathsf H_{k}^{(r)}-\zeta(r)
\end{align}for $ k\in\mathbb Z_{\geq0},r\in\mathbb Z_{>1}$. In particular, we have\begin{align}
\begin{split}
\sum_{k=0}^\infty\binom{2k}k^2\left[\mathsf H_k^{(2)}-\frac{\pi^{2}}{6}\right]\left(\frac{t }{2^4}\right)^k={}&\frac{2}{\pi}\int_0^1\frac{\mathbf K\big(\sqrt{t x}\big)\log x}{1-x}\D x=-\frac{2}{\pi}\int_0^1\frac{\big(\arccos\sqrt{x}\big)^{2}\D x}{\sqrt{x(1-x)(1-t x)}},
\end{split}
\end{align}where the last equality hinged on  fractional integration by parts (cf.\  \cite[(5.2)]{Campbell2023CCG}). With the Jacobi amplitude function $ \am(u|t)$, which  inverts the incomplete elliptic integral of the first kind, in the following sense:\begin{align}
u=\int_0^{\am(u|t)}\frac{\D\theta}{\sqrt{1-t\sin^2\theta}}\in\big[0,\mathbf K\big(\sqrt{t}\big)\big],
\end{align} we may compute\begin{align}
\int_0^1\frac{\big(\arccos\sqrt{x}\big)^{2}\D x}{\sqrt{x(1-x)(1-t x)}}=2\int_0^{\mathbf K(\sqrt{t})}\left[ \frac{\pi}{2} -\am(u|t)\right]^{2}\D u.
\end{align}With $q=e^{-\pi\mathbf K(\sqrt{1-t})/\mathbf K(\sqrt{t})} $, we have the Fourier series \cite[item 908.00]{ByrdFriedman} \begin{align}
\am(u|t)=\frac{\pi u}{2\mathbf K\big(\sqrt{t}\big)}+2\sum_{m=0}^\infty\frac{q^{m+1}}{(m+1)(1+q^{2m+2})}\sin\frac{(m+1)\pi u}{\mathbf K\big(\sqrt{t}\big)}
\end{align}at our disposal,  so a brief calculation reveals that \begin{align}
\int_0^1\frac{\big(\arccos\sqrt{x}\big)^{2}\D x}{\sqrt{x(1-x)(1-t x)}}={}&\mathbf K\big(\sqrt{t}\big)\left[ \frac{\pi^2}{6} -4\sum_{k=1}^\infty\frac{1}{k^2}\frac{q^{k}}{1+q^{2k}}\left( 1-\frac{q^{k}}{1+q^{2k}} \right)\right].
\end{align}Thus, we eventually get\begin{align}
\sum_{k=0}^\infty\binom{2k}k^2\mathsf H_k^{(2)}\left(\frac{t }{2^4}\right)^k=\frac{8\mathbf K\big(\sqrt{t}\big)}{\pi}\sum_{k=1}^\infty\frac{1}{k^2}\frac{q^{k}}{1+q^{2k}}\left( 1-\frac{q^{k}}{1+q^{2k}} \right),
\end{align}as claimed.
\end{proof}\begin{remark}We have \big[see \eqref{eq:EichlerE4'} for the definition of $ \mathscr E'_4$\big]\begin{align}
\sum_{n=0}^\infty\frac{1}{(2n+1)^{2}\cosh^2\frac{(2n+1)\pi z_{}}{i}}={}&\frac{\pi^{2}\big[4\mathscr E'_4\big(z+\frac12\big)-\mathscr E'_4(2z)\big]}{120},\label{eq:EichlerE4dP''}\\\sum_{n=1}^\infty\frac{1}{n^{2}\cosh^2\frac{2n\pi z_{}}{i}}={}&\frac{\pi^{2}\big[4\mathscr E'_4\big(4z\big)-\mathscr E'_4(2z)\big]}{30},
\end{align}but there are no straightforward ways to express the infinite series \begin{align}
\sum_{n=1}^\infty\frac{1}{n^{2}\cosh\frac{2n\pi z_{}}{i}}
\label{eq:unres}\end{align}through the values of the Eichler integrals (and/or their derivatives) at special points. In fact, the last displayed formula does not fit into an ``automorphic'' functional equation like \eqref{eq:EichlerE4refl}---a close analog  would be an entry in Ramanujan's second notebook \cite[p.~277]{RN2}:\begin{align}
\sum_{n=1}^\infty\frac{1}{n^{2}\cosh\frac{2n\pi z_{}}{i}}=\frac{\pi^{2}\big(1-6z^2\big)}{6}-\frac{8zG}{i}-\frac{16z}{i}\sum_{n=0}^\infty\frac{(-1)^{n}}{(2n+1)^{2}\Big[ e^{\frac{(2n+1)\pi i_{}}{2z}} -1\Big]}.\label{eq:RN2p277}
\end{align} Such an infinite series can be expressed through elliptic logarithms [see \eqref{eq:RN2p277'} below], yet does not resolve into any Eichler-type integrals (so far as we understand). Indeed, we do not know how to express a special value (cf.\ \cite[(2.34)]{Zhou2025NotesBHS})\begin{align}
\frac{\displaystyle\sum_{k=0}^\infty \binom{2k}k^2\frac{\mathsf H_k^{(2)}}{2^{5k}}}{\displaystyle\sum_{k=0}^\infty \binom{2k}k^2\frac{1}{2^{5k}}}=\frac{\pi^{2}}{6}-\frac{2}{9} {_4F_3}\left(\left.\begin{array}{@{}c@{}}
1,1,1,\frac{3}{2} \\[4pt]
\frac{7}{4},\frac{7}{4},2 \\
\end{array}\right| 1\right)-\frac{16\pi^{2}}{\big[\Gamma\big(\frac14\big)\big]^4}{_4F_3}\left(\left.\begin{array}{@{}c@{}}
\frac{1}{2},\frac{1}{2},1,1 \\[4pt]
\frac{5}{4},\frac{5}{4},\frac{3}{2} \\
\end{array}\right| 1\right)
\end{align}
in terms of familiar mathematical constants.\eor\end{remark}

While performing modular parametrizations for infinite series whose summands are (inversely) proportional to $ \binom{2k}{k}^2$, one often ends up with  ``unresolved'' series like \eqref{eq:unres}. This may occur even without the participation of harmonic numbers, such as (cf.\ \cite[item 908.02]{ByrdFriedman}) \begin{align}
\begin{split}\sum_{k=1}^\infty\frac{(2^4 t)^k}{k^{2}\binom{2k}k^2}={}&\int_{0}^1\frac{2\sqrt{t}\arcsin\sqrt{tx}}{\sqrt{x(1-x)(1-tx)}}\D x=2\sqrt{t}\int_0^{2\mathbf K(\sqrt{t})}\arcsin\big(\sqrt{t}\sin\am(u|t)\big)\D u\\={}&\frac{2^{5}\sqrt{t}\mathbf K\big(\sqrt{t}\big)}{\pi}\sum_{n=0}^\infty\frac{1}{(2n+1)^{2}}\frac{q^{n+\frac{1}{2}}}{1+q^{2n+1}}
\end{split}\label{eq:inv_sqr}
\end{align}for  $t\in(0,1)$ and $ q=e^{-\pi\mathbf K(\sqrt{1-t})/\mathbf K(\sqrt{t})}$. By contrast, the Eichler integral representations in Theorems \ref{thm:H2} and \ref{thm:H3} are rare occurrences.
\subsection{Infinite series involving $ \binom{2k}{k}^3$ and harmonic numbers of orders up to $3$\label{eq:Hmix}}In this subsection, we record variants of Lemma \ref{lm:H3a} and Lemma \ref{lm:H3b}, in which the summands involve harmonic numbers    $ \mathsf H_k^{(r)}\colonequals \sum_{0<n\leq k}\frac{1}{n^{r}}$ of orders $r\in\{1,2,3\}$. By convention, we will  abbreviate $\mathsf H_k^{}\equiv\mathsf H_k^{(1)}$. \begin{proposition}For  $ |4t(1-t)|\leq1$ and $ \RE t\leq\frac12$, we have \begin{align}
\begin{split}&32\sum_{k=0}^\infty\binom{2k}{k}^3\left\{ \mathsf H_{2k}^{(3)}-\frac{1}{8}\mathsf H_{k}^{(3)} -\frac{3}{2}\left[\mathsf H_{2k}^{(2)}-\frac{1}{4}\mathsf H_{k}^{(2)} \right]\left(\mathsf H_{2k} ^{}-\mathsf H_k^{}\right)\right\}\left[ \frac{t(1-t)}{2^4} \right]^{k}\\={}&28\zeta(3)[P_{-1/2}(1-2t)]^2-\pi P_{-1/2}(2t-1)\left.\!\frac{\partial^2 P_{\nu}(1-2t)}{\partial\nu^2} \right|_{\nu=-1/2}\\{}&- \pi P_{-1/2}(1-2t)\left.\!\frac{\partial^2 P_{\nu}(2t-1)}{\partial\nu^2} \right|_{\nu=-1/2}\\{}&-P_{-1/2}(1-2t)\left[ \pi^{3}P_{-1/2}(2t-1)+2\left.\!\frac{\partial^2 P_{\nu}(1-2t)}{\partial\nu^2} \right|_{\nu=-1/2}\log\frac{t(1-t)}{2^4}\right].\end{split}\label{eq:H3mix1}
\end{align}\end{proposition}\begin{proof}Alongside   \eqref{eq:Clausen'}, we have an  identity of Clausen's type  (cf.\ \cite[(3.19)--(3.20)]{Zhou2025NotesBHS})\footnote{Both \eqref{eq:Clausen'} and \eqref{eq:Clausen}
apply to $ |4t(1-t)|\leq1$ and $ \RE t\leq\frac12$. } \begin{align}
{_3}F_2\left(\left. \begin{array}{@{}c@{}}
-\nu+\varepsilon,1+\nu+\varepsilon,\frac{1}{2}+\varepsilon\ \\
1+\varepsilon,1+2\varepsilon\ \\
\end{array} \right| 4t(1-t)\right)=\frac{[\Gamma(1+\varepsilon)P^{-\varepsilon}_\nu(1-2t)]^{2}}{[t(1-t)]^{\varepsilon}}.\label{eq:Clausen}
\end{align}For  $ |4t(1-t)|\leq1$, we may turn \begin{align}\left.\!\frac{\partial^3}{\partial\nu ^2\partial\varepsilon}{_3F_2}\left(\left. \begin{array}{@{}c@{}}
-\nu+\varepsilon,1+\nu+\varepsilon,\frac{1}{2}+\varepsilon\ \\
1+\varepsilon,1+2\varepsilon\ \\
\end{array} \right| 4t(1-t)\right)\right|_{\raisebox{.3em}{\ensuremath{\substack{\nu=-1/2\\\varepsilon=0\hfill}}}}
\end{align}into the left-hand side of \eqref{eq:H3mix1} by termwise differentiation. Meanwhile, evaluating\begin{align}
\left.\!\frac{\partial^3}{\partial\nu ^2\partial\varepsilon}\frac{[\Gamma(1+\varepsilon)P^{-\varepsilon}_\nu(1-2t)]^{2}}{[t(1-t)]^{\varepsilon}}\right|_{\raisebox{.3em}{\ensuremath{\substack{\nu=-1/2\\\varepsilon=0\hfill}}}}\end{align} with the aid of \cite[Lemma 2.2]{Zhou2025NotesBHS}, we reach  the  right-hand side of \eqref{eq:H3mix1}.
\end{proof}\begin{remark} The identity\begin{align}
\left.\!\frac{1}{P_{-1/2}(1-2\alpha_{4}(z))}\frac{\partial^2 P_{\nu}(1-2\alpha_{4}(z))}{\partial\nu^2} \right|_{\nu=-1/2}=\frac{\pi^{2}\big[4\mathscr E'_4\big(z+\frac12\big)-\mathscr E'_4(2z)\big]}{15}\tag{\ref{eq:EichlerE4dP''}$'$}
\end{align}  may help us deduce a modular parametrization for \eqref{eq:H3mix1}, but this does not lead to closed-form evaluations (in terms of well-known mathematical constants), due to lack of ``magic cancellations'' similar to those encountered in   \S\ref{sec:H2} and \S\ref{sec:H3}.
\eor\end{remark}

\begin{proposition}For  $ |4t(1-t)|\leq1$, $ \RE t\leq\frac12$, and $ \I\frac{1}{t(1-t)}\neq0$, we have \begin{align}\footnotesize
\begin{split}&4\sum_{k=0}^\infty\binom{2k}{k}^3\left[ \mathsf H_{k}^{(3)} -3\mathsf H_{k}^{(2)} \left(\mathsf H_{2k} ^{}-\mathsf H_k^{}\right)\right]\left[ \frac{t(1-t)}{2^4} \right]^{k}\\={}&-\left(\frac{2}{\pi}\right)^2\int_t^\infty\frac{4(1-2s)}{s(1-s)}\big[\mathbf K\big(\sqrt{\mathstrut 1-s}\big)\mathbf K\big(\sqrt{\mathstrut t}\big)-\mathbf K\big(\sqrt{\mathstrut s}\big)\mathbf K\big(\sqrt{\mathstrut 1-t}\big)\big]^{2}\D s\\&{}+\frac{\sqrt{-T}P_{-1/2}\big(\sqrt{1-T}\big) P_{-1/2}\big(\!-\!\sqrt{1-T}\big)}{3}\left\{ \pi ^{3}-\frac{i\I T}{|\I T|}\left[\pi^2\log(-64T)-12\zeta(3)\right] \right\}\\&{}+\frac{2\sqrt{-T}}{3}\left\{ \left[P_{-1/2}\big(\sqrt{1-T}\big)\right] ^{2}-\left[ P_{-1/2}\big(\!-\!\sqrt{1-T}\big) \right]^2\right\}\left[\pi^2\log(-64T)-3\zeta(3)\right]\\&{}-\frac{2\pi ^{3}\sqrt{-T}}{3}\frac{i\I T}{|\I T|}\left[P_{-1/2}\big(\sqrt{1-T}\big)\right]^{2}\\&{}+\sqrt{-T}\left\{ P_{-1/2}\big(\sqrt{1-T}\big) \left[ \pi-\frac{i\I T}{|\I T|} \log(-64T)\right] -P_{-1/2}\big(\!-\!\sqrt{1-T}\big)\log(-64T)\right\}\left.\!\frac{\partial^2 P_{\nu}\big(\!-\!\sqrt{1-T}\big)}{\partial\nu^2} \right|_{\nu=-1/2}\\&{}+\sqrt{-T}\left\{ P_{-1/2}\big(\!-\!\sqrt{1-T}\big) \left[ \pi-\frac{i\I T}{|\I T|} \log(-64T)\right] \right.\\&{}\left.{}+P_{-1/2}\big(\sqrt{1-T}\big)\left[\log(-64T)+\frac{2\pi i\I T}{|\I T|}\right]\right\}\left.\!\frac{\partial^2 P_{\nu}\big(\sqrt{1-T}\big)}{\partial\nu^2} \right|_{\nu=-1/2},\end{split}\label{eq:H3mix2}
\end{align}where $ T=\frac{1}{4t(1-t)}$.
\end{proposition}\begin{proof}For $\nu\in(-1,0)$,  $ \varepsilon\in(0,1)$, $|T|=1 $ and $\I T\neq0$, we have  the following functional equation (cf.\ \cite[Lemma 1]{GuilleraRogers2014}):\begin{align}\begin{split}&
{_4F_3}\left(\left. \begin{array}{@{}c@{}}
1, 1-\varepsilon,1-\varepsilon,1-\varepsilon \\
\frac{3}{2}-\varepsilon,1-\nu-\varepsilon,2+\nu-\varepsilon \\
\end{array} \right| \frac{1}{T}\right){}\\={}&-\frac{(1-2\varepsilon ) (\nu +\varepsilon) (1+\nu -\varepsilon )T }{2\varepsilon ^3}{_4F_3}\left(\left. \begin{array}{@{}c@{}}
1, -\nu+\varepsilon,1+\nu+\varepsilon,\frac{1}{2}+\varepsilon \\
1+\varepsilon,1+\varepsilon,1+\varepsilon \\
\end{array} \right| T\right)\\&{}+\frac{ (-T)^{1-\varepsilon }  \pi \sqrt{\pi }\Gamma \big(\frac{3}{2}-\varepsilon \big) \Gamma (1-\nu-\varepsilon) \Gamma (2+\nu-\varepsilon)}{2 [\Gamma (1-\varepsilon )]^3\sin (\varepsilon \pi  )}\left\{ \left[\frac{1}{\sin(\nu\pi)}-\frac{2\sin (\nu \pi  )}{\tan^2(\varepsilon\pi)}\right. \right.\\{}&\left.{}+\frac{2i\sin (\nu \pi  )}{\tan(\varepsilon\pi)}\frac{\I T}{|\I T|}\right]\big[P_{\nu}\big(\sqrt{1-T}\big)\big]^{2}-\frac{\big[P_{\nu}\big(\!-\!\sqrt{1-T}\big)\big]^{2}}{\sin (\nu \pi  )}
\\{}&\left.{}-2\left[\frac{ 1}{\tan(\varepsilon\pi)} -\frac{i\I T}{|\I T|}\right]P_{\nu}\big(\sqrt{1-T}\big)P_{\nu}\big(\!-\!\sqrt{1-T}\big)\right\}.\end{split}\label{eq:GRrecip}\end{align}
Here, while counting the residues of  \begin{align}
\frac{ \Gamma \big(\frac{3}{2}-\varepsilon \big) [\Gamma (1-s-\varepsilon )]^3 \Gamma (1 -\nu-\varepsilon ) \Gamma (2 +\nu -\varepsilon)}{[\Gamma (1-\varepsilon )]^3 \Gamma \big(\frac{3}{2}-s-\varepsilon \big) \Gamma (1-s-\nu-\varepsilon  ) \Gamma (2-s +\nu -\varepsilon)}\frac{\pi(-T)^{s}}{\sin(s\pi)}
\end{align} in the complex $s$-plane, we have expanded \eqref{eq:Clausen'} and \eqref{eq:Clausen} up to $O(\varepsilon^2)$ terms, effectively expressing certain  series (cf.\ \cite[(3.16)--(3.17)]{Zhou2025NotesBHS}) in terms of $ \big[P_{\nu}\big(\sqrt{1-T}\big)\big]^{2}$,  $P_{\nu}\big(\sqrt{1-T}\big)P_{\nu}\big(\!-\!\sqrt{1-T}\big)$ and $\big[P_{\nu}\big(\!-\!\sqrt{1-T}\big)\big]^{2} $.

In the original setting of \cite[Lemma 1]{GuilleraRogers2014}, the parameter $ \varepsilon$ was intended as an infinitesimal, but this does not stop us from considering\begin{align}
\left.\!\frac{\partial^3}{\partial\nu ^2\partial\varepsilon}{_4F_3}\left(\left. \begin{array}{@{}c@{}}
1, 1-\varepsilon,1-\varepsilon,1-\varepsilon \\
\frac{3}{2}-\varepsilon,1-\nu-\varepsilon,2+\nu-\varepsilon \\
\end{array} \right| 4t(1-t)\right)\right|_{\raisebox{.3em}{\ensuremath{\substack{\nu=-1/2\\\varepsilon=1/2\hfill}}}},
\end{align} which is equal to the left-hand side of \eqref{eq:H3mix2}. Now we apply the same differential operator to the right-hand side of  \eqref{eq:GRrecip} for $ |4t(1-t)|=1$. First, note that  [cf.\ \eqref{eq:4F3Rama}]\begin{align}\begin{split}&
\left.\!\frac{\partial^3}{\partial\nu ^2\partial\varepsilon}\left[-\frac{(1-2\varepsilon ) (\nu +\varepsilon) (1+\nu -\varepsilon )T }{2\varepsilon ^3}{_4F_3}\left(\left. \begin{array}{@{}c@{}}
1, -\nu+\varepsilon,1+\nu+\varepsilon,\frac{1}{2}+\varepsilon \\
1+\varepsilon,1+\varepsilon,1+\varepsilon \\
\end{array} \right| T\right)\right]\right|_{\raisebox{.3em}{\ensuremath{\substack{\nu=-1/2\\\varepsilon=1/2\hfill}}}}\\={}&-\left(\frac{2}{\pi}\right)^2\int_t^\infty\frac{4(1-2s)}{s(1-s)}\big[\mathbf K\big(\sqrt{\mathstrut 1-s}\big)\mathbf K\big(\sqrt{\mathstrut t}\big)-\mathbf K\big(\sqrt{\mathstrut s}\big)\mathbf K\big(\sqrt{\mathstrut 1-t}\big)\big]^{2}\D s
\end{split}
\end{align}enters the right-hand side of \eqref{eq:H3mix2}. Next, the third-order derivative of the  remaining terms on the right-hand side of  \eqref{eq:GRrecip} are responsible for all the expressions involving $ T=\frac{1}{4t(1-t)}$.

Although the computations above depend on a constraint  $ |4t(1-t)|=1$, the generic form of   \eqref{eq:H3mix2} follows readily from analytic continuation.          \end{proof}\begin{remark}One may introduce modular parametrizations to   \eqref{eq:GRrecip}, while observing that (cf.\ \cite[pp.\ 38--40]{ByrdFriedman})\begin{align}
2^{4}\alpha_{4}(z)[1-\alpha_4(z)]\alpha_4\left( \frac{1}{2} \frac{2z-1}{2z+1}\right)\left[ 1-\alpha_4\left( \frac{1}{2} \frac{2z-1}{2z+1}\right) \right]={}&1.
\end{align}The main lesson here is the following: even when the summands of our series are not as fine-tuned as those in Theorem \ref{thm:H2} or Theorem  \ref{thm:H3}, there may still be modularity lurking somewhere, though perhaps in a very sophisticated manner.
\eor\end{remark}
\subsection{A perspective on binomial harmonic sums}In recent years, the numerical experiments of the first-named author \cite{Sun2024binomI,Sun2022,Sun2026binomII} have led to many conjectural closed forms for  \textit{binomial harmonic sums}  in  the form of \begin{align}
\sum_{k=1}^\infty\frac{\binom{2k}k^{\mathscr N}}{k^s(2k-1)^{s'}}\left( \frac{T}{2^{2\mathscr N}} \right)^{k}\left[ \prod_{j=1}^M\mathsf H_{k}^{(r_j)} \right]\left[\prod_{\vphantom{j}\smash[b]{j'}=1}^{M'}{{\mathsf H}}_{\smash[t]{2}k}^{( {r}'_{j'})}\right],\label{eq:binomH'}
\end{align} where $\mathscr N,s,s'\in\mathbb Z
$, $ M,M'\in\mathbb Z_{\geq0}$, and $ |T|<1$.

The  $\mathscr N\in\{-1,0,1\}$  cases are pertinent to  $ \varepsilon$-expansions [namely, hypergeometric deformations of the   $ _pF_q$ series defined in  \eqref{eq:defn_pFq}] of quantum electrodynamics and quantum chromodynamics \cite{Ablinger2017,Ablinger2022,ABRS2014,ABS2011,ABS2013,BorweinBroadhurstKamnitzer2001,Broadhurst1996,Broadhurst1999,DavydychevKalmykov2001,DavydychevKalmykov2004,KalmykovVeretin2000,Kalmykov2007,Weinzierl2004bn}.
In  \cite[\S3]{Zhou2022mkMpl},  binomial harmonic sums for  $\mathscr N\in\{-1,0,1\}$ are converted into linear combinations of contour integrals \begin{align}
\int_C \Li_{a_1,\dots,a_n}(z_1(t),\dots,z_n(t))\frac{\D t}{R(t)},
\label{eq:MPL_int_repn}\end{align}before their closed forms are found. Here in the integrand of \eqref{eq:MPL_int_repn}, we have rational functions   $ z_1(t),\dots,z_n(t)$, and $R(t)$,   while  Goncharov's multiple polylogarithms (MPLs) \cite{Goncharov1997,Goncharov1998} are analytic continuations of the convergent series\begin{align}
\Li_{a_1,\dots,a_n}(z_1,\dots,z_n)\colonequals \sum_{\ell_{1}>\dots>\ell_{n}>0}\prod_{j=1}^n\frac{z_{j}^{\ell_{j}}}{\ell_j^{a_j}}\label{eq:Mpl_defn}\end{align}
 for certain $ a_1,\dots,a_n\in\mathbb Z_{>0}$. All these MPLs themselves arise from (iterated) integrals over products of logarithms and rational functions, such as  \begin{align}
\int_0^z\frac{\log (1-t)}{1-t}\log t\D t=\Li_{2,1}(z,1)-\frac{\log^{2} (1-z)}{2}\log z
\end{align}for $ 0<|z|<1$.
Integrals over MPLs, like \eqref{eq:MPL_int_repn}, are also expressible through a linear combination of MPLs, which in turn, leads to closed-form evaluation of    \eqref{eq:binomH'} for   $\mathscr N\in\{-1,0,1\}$. For $z_1=\dots =z_n=1$ and $ a_1\neq1$, the MPL $ \Li_{a_1,\dots,a_n}(z_1,\dots,z_n)$ becomes a multiple zeta value (MZV) $ \zeta_{a_1,\dots,a_n}$, which generalizes  the Riemann zeta value $\zeta(a)\equiv \zeta_a $ for $ a\in\mathbb Z_{>1}$. From a higher point of view, the MPLs and the MZVs are essentially motivic periods on the moduli spaces of genus-zero curves (namely, rational curves) \cite{Brown2009a,Brown2009arXiv,Brown2009b}.

If we have  $ |\mathscr N|>1$ in   \eqref{eq:binomH'},\footnote{For the borderline cases where $ T=1$ and $ \mathscr N\in\{-2,2\}$ in   \eqref{eq:binomH'}, one is still dealing with rational curves instead of elliptic curves. The corresponding series have been already covered by the methods  in \cite[\S3]{Zhou2022mkMpl}.} then we are obliged to cope with the  moduli spaces of genus-one curves (namely, elliptic curves), in which  the complete elliptic integrals $ \mathbf K\big(\sqrt{t})$ and $ \mathbf K\big(\sqrt{1-t}\big)$ take the places of $\log t$ and $\log(1-t)$.
There is a wealth of  literature on genus-one analogs of MZVs/MPLs, namely elliptic multiple zeta values (EMZVs) \cite{BMMS2015EMZV,BSZ2019,Enriquez2014EMZV,Enriquez2016EMZV,Goncharov1998ellmot,LMS2021EMZV,Nils2020EMZV,ZerbiniThesis} and elliptic multiple polylogarithms (EMPLs) \cite{ABDvHIRRS2018,ABSW2016kiteEPL,AW2018EMPL,AdamsWeinzierl2019,BHMcLvHW2018tt,BDDPT2018EPL,BDDT2018EMPL,BDDT2018EMPL2,BroedelKaderli2020EPL,BrownLevin2013EMPL,Samart2020,WZ2023EMPL}, with diverse notational conventions. In one of these notational systems (see \cite{ABDvHIRRS2018,ABSW2016kiteEPL,AW2018EMPL,BroedelKaderli2020EPL}, for example),  {the elliptic polylogarithm} is defined as\begin{align} \ELi_{n;m}(x;y;q)\colonequals \sum_{j=1}^\infty\sum_{k=1}^\infty\frac{x^jy^k}{j^{n}k^m}q^{jk}=\sum_{j=1}^\infty\frac{x^{j}}{j^{n}}\Li_m\big(y q^j\big)=\sum_{k=1}^\infty\frac{y^{k}}{k^{m}}\Li_n\big(x q^k\big).\label{eq:Eli_defn}
\end{align}This definition   not only echoes   \eqref{eq:Mpl_defn}, but also embodies  the Lambert--Ramanujan series\begin{align}
\ELi_{0;m}(1;1;q)={}&\sum_{k=1}^\infty\frac{1}{k^m}\frac{q^k}{1-q^{k}}
\end{align}as a special case. Since we have \begin{align}
{}&\sum_{k=0}^\infty\frac{(-1)^{k}}{(2k+1)^2}\frac{1}{q^{-2k-1}-1}=\frac{8\ELi_{0;2}(1;i;q)+2\ELi_{0;2}(1;1;q^{2})-\ELi_{0;2}(1;1;q^{4})}{8i},\label{eq:RN2p277'}
\end{align}the right-hand side of \eqref{eq:RN2p277} can be expressed in terms of elliptic polylogarithms.
It appears that for generic settings of   \eqref{eq:binomH'} with   $ |\mathscr N|>1$,  one cannot reduce the corresponding elliptic polylogarithms [cf.\ \eqref{eq:unres}  and \eqref{eq:inv_sqr}] to familiar arithmetic objects like Eichler integrals.

According to   \cite{Zhou2025NotesBHS} and the current work, each provable evaluation of  \eqref{eq:binomH'} for $ \mathscr N\in\{-3,2,3\}$ invokes a product of at least  $ |\mathscr N|-1$ elliptic integrals (either in isolation or as part of an  integrand).
It is possible to derive  closed-form evaluations for certain integrals over the products of four or more elliptic integrals, such as (cf.\ \cite[Table 1]{WanZucker2014}, \cite[(4.18)--(4.19)]{EZF}, and \cite[(1.1)]{Zhou2017Int4Pnu})\begin{align}
\zeta(5)={}&\frac{8}{93}\int_0^1(1-2t)\big[\mathbf K\big(\sqrt{1-t}\big)\big]^4\D t,\label{eq:zeta5_int}\\\zeta(7)={}&\frac{32}{5715}\int_0^1[2-17t(1-t)]\big[\mathbf K\big(\sqrt{1-t}\big)\big]^6\D t.
\end{align}More generally, some special values of the Epstein zeta functions $ E(z,s)$ and the Dirichlet $ L$-functions $ L_d(s)$ (where $ s\in\mathbb Z_{>1}$)  can be represented by integrals whose integrands involve products of $ 2(s-1)$ complete elliptic integrals, such as (cf.\ \cite[(4.29)]{EZF})\begin{align}
\begin{split}&\frac{105L_{-4}(4)}{136 \pi ^4}\\={}&\frac{200025 \zeta (7)}{2176 \pi ^7}-\frac{70}{136 \pi ^7}\int_0^{1/2}[2-17t(1-t)]\big[\mathbf K\big(\sqrt{t}\big)\big]^6\left\{\left[ \frac{\mathbf K\big(\sqrt{1-t}\big)}{\mathbf K\big(\sqrt{t}\big)} \right]^{2}-1\right\}^{3}\D t.
\end{split}
\end{align} Will these integral formulae  shed new light\footnote{Some individual cases of  \eqref{eq:binomH'}  for $ \mathscr N\in\{-7,-5,5,7\}$ have already been  tackled combinatorially \cite{Au2025a,Au2025b,HouHeWang2023,Wei2023c,Wei2023Sun,Wei2023b,WeiXu2023}. As for ``new light'', we are expecting modular parametrizations similar to the current work. } on   \eqref{eq:binomH'}  for $ \mathscr N\in\{-7,-5,5,7\}$? Maybe we will see the answer in the near future.



\end{document}